\documentclass{article}
\usepackage{amssymb,amsmath,bm,geometry,graphics,graphicx,url,color}

\def\no{\noindent}
\def\pmatrix{\left(\begin{array}}
\def\endpmatrix{\end{array}\right)}

\newtheorem{theo}{Theorem}
\newtheorem{lem}{Lemma}

\newtheorem{defi}{Definition}

\newtheorem{assum}{Assumption}
\newtheorem{prop}{Proposition}

\newcommand{\abs}[1]{\left\vert#1\right\vert}

\newcommand{\norm}[1]{\left\Vert#1\right\Vert}

\title{Long-time  momentum and actions behaviour of
energy-preserving methods  for semilinear wave equations via spatial
spectral semi-discretizations}

\author{Bin Wang\,
\footnote{School of Mathematical Sciences, Qufu Normal University,
Qufu 273165, P.R. China; Mathematisches Institut, University of
T\"{u}bingen, Auf der Morgenstelle 10, 72076 T\"{u}bingen, Germany.
The research is supported in part by the Alexander von Humboldt
Foundation and by the Natural Science Foundation of Shandong
Province (Outstanding Youth Foundation) under Grant ZR2017JL003.
E-mail:~{\tt wang@na.uni-tuebingen.de} } \and Xinyuan
Wu\thanks{School of Mathematical Sciences, Qufu Normal University,
Qufu 273165, P.R. China; Department of Mathematics, Nanjing
University, Nanjing 210093, P.R. China. The research is supported in
part by the National Natural Science Foundation of China under Grant
11671200. E-mail:~{\tt xywu@nju.edu.cn}} }

\begin{document}
\maketitle

\begin{abstract} As is known that wave equations have physically very important
properties which should be respected by numerical schemes in order
to predict correctly the solution over a long time period. In this
paper, the long-time behaviour of momentum and actions for
energy-preserving methods is analysed for semilinear wave equations.
A full discretisation of wave equations is derived and analysed by
firstly using  a spectral semi-discretisation in space and then by
applying the adopted average vector field (AAVF) method in time.
This numerical scheme can exactly preserve the energy of the
semi-discrete system. The main theme of this paper is to analyse
another important physical property of the scheme. It is shown that
this scheme yields near conservation of a modified momentum and
modified actions over long times. Both the results are rigorously
proved based on the technique of modulated Fourier expansions in two
stages. First a multi-frequency modulated Fourier expansion of the
AAVF method is constructed and then two almost-invariants of the
modulation system are derived.
\medskip

\no{\bf Keywords:} Semilinear wave equations\and Energy-preserving
methods\and Multi-frequency modulated Fourier expansion\and Momentum
and actions conservation

\medskip
\no{\bf MSC:}35L70\and  65M70\and   65M15

\end{abstract}

\section{Introduction}
This paper is concerned with the long-time behaviour of
energy-preserving (EP) methods  when applied to the following
one-dimensional semilinear wave equation (see
\cite{Cohen08-1,Cohen08,Hairer08})
\begin{equation}\label{wave equa}
\begin{array}[c]{ll}
u_{tt}-u_{xx}+\rho u+g(u)=0,\ \ \ t>0,\ \ -\pi\leq x\leq \pi,
\end{array}
\end{equation}
where $g$  is a  nonlinear and smooth real function with $g(0) =
g'(0) = 0$ and $\rho$ is a real number satisfying $\rho> 0$.
Similarly to the refs. \cite{Cohen08-1,Cohen08,Hairer08}, the
initial values $u(\cdot, 0)$ and $u_t (\cdot, 0)$ for this equation
are assumed to be bounded by a small parameter $\epsilon$, which
provides  small initial data  in appropriate Sobolev norms.
Meanwhile,  periodic boundary conditions are   considered in this
paper.

It is noted  that several important quantities are conserved by the
solution  of \eqref{wave equa}.    The total energy
\begin{equation*}
\begin{aligned}
H(u,v)=\frac{1}{ 2 \pi}\int_{-\pi}^{\pi}\Big(
\frac{1}{2}\big(v^2+(\partial_{x}u)^2+\rho
u^2\big)(x)+U(u(x))\Big)dx
\end{aligned}
\end{equation*}
is   exactly preserved along the solution,    where    $v
=\partial_{t}u$ and the potential $U(u)$ is of the form $U'(u) =
g(u)$. The solution of \eqref{wave equa} also conserves the momentum
\begin{equation*}
\begin{aligned}
K(u,v)=\frac{1}{ 2
\pi}\int_{-\pi}^{\pi}\partial_{x}u(x)v(x)dx=-\sum\limits_{k=-\infty}^{\infty}\mathrm{i}ju_{-j}v_{j},
\end{aligned}
\end{equation*}
  where $u_j $ and $v_j$ are  the Fourier coefficients in the series $u(x)
=\sum\limits_{j=-\infty}^{\infty}u_je^{\mathrm{i}jx}$ and
 $v(x)
=\sum\limits_{j=-\infty}^{\infty}v_je^{\mathrm{i}jx}$, respectively.
 The harmonic actions $$ I_j(u,v)=\frac{\omega_j}{ 2}|u_j|^2+
\frac{1}{ 2\omega_j}|v_j|^2 $$ are conserved for the linear wave
equation, where $\omega_j=\sqrt{\rho+j^2}$ for $j\in \mathbb{Z} $.
For the nonlinear case, it has been proved in
\cite{Bambusi03,Cohen08} that the actions remain constant up to
small deviations over a long time period for   smooth and small
initial data and almost all values of $\rho> 0$.

It has now become a common practice that the consideration of
qualitative properties in ordinary and more recently in partial
differential equations is important when designing numerical schemes
in the sense of structure preservation. In recent decades many
numerical methods have been developed and researched for solving
wave equations (see, e.g.
\cite{Cano13,Cano06,Cano14,Gauckler15,Gauckler17-1,Gauckler18,Grimm06,Liu_Iserles_Wu(2017-2),wang-2016,wubook2018}).
 As one important aspect of the analysis, long-time conservation properties of  wave equations or
of some numerical methods  applied to wave equations have been well
studied and we refer the reader to
 \cite{Cohen08,Cohen08-1,Gauckler16,Gauckler12,Hairer08}. All these analyses are achieved by the technique of
modulated Fourier expansions, which was developed by Hairer and
Lubich in \cite{Hairer00}  and   has been frequently used in the
long-term analysis (see, e.g.
\cite{Cohen12,Hairer16,hairer2006,McLachlan14,Sanz-Serna09}). On the
other hand, as an important kind of methods, energy-preserving (EP)
methods have also been the subject of many investigations for wave
equations. EP methods can exactly preserve the energy of the
considered system. With regard to some examples of this topic, we
refer the reader to
\cite{r5,Celledoni12,Li_Wu(JCP2016),JMAA(2015)_Liu_Wu,kai-2017,CiCP(2017)_Mei_Liu_Wu,IMA2018}.
Unfortunately however, it seems that the long-time behaviour of EP
methods in other structure-preserving aspects has  not been studied
for wave equations  in the literature, such as the numerical
conservation of momentum and actions.

The  main contribution of this paper is to rigorously analyse  the
long-time momentum and actions conservations of EP methods for wave
equations.   To our knowledge, this is the first research that
studies the   long-time behaviour of EP methods on wave equations by
using modulated Fourier expansions. We organise the rest of this
paper  as follows. A full discretisation of the semilinear wave
equation \eqref{wave equa} by using spectral semi-discretisation in
space and EP methods in time is given in Sect. \ref{sec Spectral
semi-dis}. The main result of this paper is presented in Sect.
\ref{sec:Main results} and a numerical experiment is carried out to
support  the theoretical result. The proof of the main result is
given in detail in   Sect. \ref{sec:proof}, where the modulated
Fourier expansion of EP methods is constructed and two
almost-invariants of the modulated system are studied. Some
conclusions and further discussions are included in Sect.
\ref{sec:conclusions}.

\section{Full discretisation}\label{sec Spectral semi-dis}

\subsection{Spectral semi-discretisation in space}
We first discretise the wave equation in space by using a spectral
semi-discretisation introduced in  \cite{Cohen08-1,Hairer08}. Choose
equidistant collocation points $x_k = k\pi /M$ (for $k
=-M,-M+1,\ldots,M-1$) for the pseudo-spectral semi-discretisation in
space and consider the   real-valued trigonometric polynomials as an
approximation for the solution of \eqref{wave equa}
\begin{equation}\label{trigo pol}
\begin{array}[c]{ll}
 u^{M}(x,t)=\sum\limits_{|j|\leq M}^{'}
 q_j(t)\mathrm{e}^{\mathrm{i}jx},\quad  v^{M}(x,t)=\sum\limits_{|j|\leq M}^{'}
 p_j(t)\mathrm{e}^{\mathrm{i}jx},
\end{array}
\end{equation}
where  $p_j (t) = \frac{d}{dt}q_j (t)$ and the prime indicates that
the first and last terms in the summation are taken with the factor
$1/2$. We collect all the $q_j $ in a  $2M$-periodic coefficient
vector $\mathbf{q}(t) = (q_j (t))$, which is a solution of the
$2M$-dimensional system of oscillatory ODEs
\begin{equation}
\frac{d^2 \mathbf{q}}{dt^2}+\Omega^2 \mathbf{q}=f(\mathbf{q}), \label{prob}%
\end{equation}
where  $\Omega$ is diagonal with entries $\omega_j$ for $|j|\leq M$
and
$f(\mathbf{q})=-\mathcal{F}_{2M}g(\mathcal{F}^{-1}_{2M}\mathbf{q}),$
and $\mathcal{F}_{2M}$ denotes the discrete Fourier transform
$(\mathcal{F}_{2M}w)_j=\frac{1}{2M}\sum\limits_{k=-M}^{M-1}w_k\mathrm{e}^{-\mathrm{i}jx_k}$
 for $|j|\leq M.$
It is seen that the system \eqref{prob} is   a finite-dimensional
complex Hamiltonian system with the energy
\begin{equation}\label{H}H_M(\mathbf{q},\mathbf{p})=\frac{1}{ 2}\sum\limits_{|j|\leq
M}^{'}\big(
|p_j|^2+\omega_j^2|q_j|^2\big)+V(\mathbf{q}),\end{equation} where
$V(\mathbf{q})=\frac{1}{2M}\sum\limits_{k=-M}^{M-1}U((\mathcal{F}^{-1}_{2M}q)_k).$
The actions (for $|j|\leq M$) and the momentum of \eqref{prob}
 respectively read
$$ I_j(\mathbf{q},\mathbf{p})=\frac{\omega_j}{ 2}|q_j|^2+ \frac{1}{
2\omega_j}|p_j|^2,\ \ K(\mathbf{q},\mathbf{p})=-\sum\limits_{|j|\leq
M}''\mathrm{i}jq_{-j}p_{j},$$
 where the double prime indicates that the first and
last terms in the summation are taken with the factor $1/4$. We are
interested in real approximation \eqref{trigo pol} throughout this
study and thus it holds that
  $q_{-j} = \bar{q}_j$,  $p_{-j} =
\bar{p}_{j}$ and $I_{-j} = I_j$.

 It is noted that the energy \eqref{H} is exactly preserved along the solution of  \eqref{prob}.
 For the momentum  and actions in the
semi-discretisation, the following results have been proved in
\cite{Hairer08}.

\begin{theo} \label{08 theo 1}  (See \cite{Hairer08}.)
 Under the non-resonance condition (9) and    the assumption (16)  given in
 \cite{Hairer08}, it holds that
\begin{equation*}
\begin{aligned}
&\sum\limits_{l=0}^{M}\omega_l^{2s+1}\frac{|I_l(\mathbf{q}(t),\mathbf{p}(t))-I_l(\mathbf{q}(0),\mathbf{p}(0))|}{\epsilon^2}
\leq
C  \epsilon,\\
&\frac{|K(\mathbf{q}(t),\mathbf{p}(t))-K(\mathbf{q}(0),\mathbf{p}(0))|}{\epsilon^2}
\leq C  t \epsilon M^{-s+1},
\end{aligned}
\end{equation*}
where $0\leq t\leq \epsilon^{-N+1}$ and the  constant $C$ is
independent of $\epsilon, M, h$ and $t$.
\end{theo}

\subsection{EP methods in time}

\begin{defi}
\label{defAAVF}  (See \cite{wang2012-PLA,wu2013-JCP}.) For
efficiently solving the  oscillatory system \eqref{prob}, the
adopted average vector field (AAVF)  method has been developed,
which is defined as
\begin{equation}\left\{
\begin{aligned}
 \mathbf{q}_{n+1}&=\phi_0(V)\mathbf{q}_{n}+h\phi_1(V)\mathbf{p}_{n}+h^2\phi_2(V)\displaystyle\int_{0}^{1}f((1-\sigma)\mathbf{q}_{n}+\sigma \mathbf{q}_{n+1})d\sigma,\\
\mathbf{p}_{n+1}&=-h\Omega^2\phi_1(V)\mathbf{q}_{n}+\phi_0(V)\mathbf{p}_{n}+h\phi_1(V)\displaystyle\int_{0}^{1}f((1-\sigma)\mathbf{q}_{n}+\sigma
\mathbf{q}_{n+1})d\sigma,
\end{aligned}\right.\label{AAVFmethod}%
\end{equation}
where $h$ is the stepsize, and \begin{equation}
\phi_{l}(V):=\sum\limits_{k=0}^{\infty}\dfrac{(-1)^{k}V^{k}}{(2k+l)!},\
\ l=0,1,2  \label{Phi01}
\end{equation} are matrix-valued functions of
$V=h^{2}\Omega^2$.
\end{defi}

It follows from  \eqref{Phi01} that
\begin{equation*}
\phi_{0}(V)=\cos(h\Omega),\quad  \phi_{1}(V) = \sin(h\Omega) (h\Omega)^{-1},\quad  \phi_{2}(V) = (I-\cos(h\Omega)) (h\Omega)^{-2}.
\end{equation*}
We note that this method \eqref{AAVFmethod}  reduces to  the well
known  average vector field (AVF)  method when $V=0$.  The following
properties of the AAVF method have been shown in
\cite{wang2012-PLA,wu2013-JCP}.
\begin{prop}\label{prop h} (See \cite{wang2012-PLA,wu2013-JCP}.)
The AAVF method is symmetric and exactly preserves the   energy
\eqref{H}, which means that
$$H_M(\mathbf{q}_{n+1},\mathbf{p}_{n+1})=H_M(\mathbf{q}_{n},\mathbf{p}_{n})\quad \textmd{for}\quad n=0,1,\ldots.$$

\end{prop}

\noindent Clearly, the energy-preserving AAVF method does not
exclude symmetry structure, and as is known that preserving the
energy and symmetries of the systems at the discrete level is
important for geometry integrators.

\section{Main result and numerical experiment} \label{sec:Main results}

 \subsection{Notations} In this paper, we take the following   notations, which have been used in \cite{Cohen08-1}.
  For  sequences of integers
$\mathbf{k} = (k_l)_{l=0}^{M}$,      $\boldsymbol{\omega} =
(\omega_l)_{l=0}^{M}$, and a real $\sigma$, denote
\begin{equation*}
\begin{aligned}
|\mathbf{k}| = (|k_l|)_{l=0}^{\infty},\quad
\norm{\mathbf{k}}=\sum\limits_{l=0}^{M}|k_l|, \quad \
\mathbf{k}\cdot
\boldsymbol{\omega}=\sum\limits_{l=0}^{M}k_l\omega_l, \quad \
\boldsymbol{\omega}^{\sigma |\mathbf{k}|}=\Pi_{l=0}^{M}
\omega_l^{\sigma |k_l|}.
\end{aligned}
\end{equation*}
 Denote  by  $\langle j\rangle$  the unit coordinate vector $(0, \ldots , 0, 1, 0,
\ldots,0)^{\intercal}$  with the only entry $1$ at the $|j|$-th
position. For $s\in \mathbb{R}^+$,  the Sobolev space of
$2M$-periodic sequences $\mathbf{q}=(q_j)$ endowed with the weighted
norm $\norm{\mathbf{q}}_s=\Big(\sum\limits_{|j|\leq
M}''\omega_j^{2s} |q_j|^2\Big)^{1/2}$ is   denoted by $H^{s}$.
Moreover, we  set $ [[\mathbf{k}]]=\left\{\begin{aligned} &
(\norm{\mathbf{k}}+1)/2,\quad
\mathbf{k}\neq\textbf{0},\\
&3/2,\qquad\quad\quad \ \  \mathbf{k}=\textbf{0}.
\end{aligned}\right.
$

 \subsection{Main result}
Before presenting the main result of this paper,    the following
assumptions are needed (see \cite{Cohen08-1}).
\begin{assum}\label{ass}
 $\bullet$ It is assumed that
the initial values  of \eqref{prob} are bounded by
\begin{equation}\label{initi cond}
\big(\norm{\mathbf{q}(0)}_{s+1}^2+\norm{\mathbf{p}(0)}_{s}^2\big)^{1/2}\leq
\epsilon
\end{equation}
with a small parameter $\epsilon>0$.

 $\bullet$   The non-resonance condition is
 considered for a given stepsize $h$:
\begin{equation}
\abs{\sin\big(\frac{h}{2}(\omega_j-\mathbf{k}\cdot\boldsymbol{\omega})\big)\cdot
\sin\big(\frac{h}{2}(\omega_j+\mathbf{k}\cdot\boldsymbol{\omega})\big)}
\geq \epsilon^{1/2}h^2(
 \omega_j+|\mathbf{k} \cdot\boldsymbol{\omega}|).
 \label{inequa}%
\end{equation}
  If this is violated, we   define a  set of
near-resonant indices
\begin{equation}
\mathcal{R}_{\epsilon,h}=\{(j, \mathbf{k}):|j|\leq M,\
\norm{\mathbf{k}}\leq2N,\ \ \mathbf{k}\neq\pm\langle j\rangle,\
\textmd{not}\ \textmd{satisfying} \ \eqref{inequa}\},
 \label{near-resonant R}%
\end{equation}
where   $N \geq 1$ is    the truncation number of the expansion
\eqref{MFE-AAVF} which will be presented in the next section. We
make the following assumption for this set.
 Suppose that there exist  $\sigma
> 0$ and a constant $C_0$ such that
\begin{equation}
\sup_{(j, \mathbf{k})\in\mathcal{R}_{\epsilon,h}}
\frac{\omega_j^{\sigma}}{\boldsymbol{\omega}^{\sigma
|\mathbf{k}|}}\epsilon^{\norm{\mathbf{k}}/2}\leq C_0 \epsilon^N.
 \label{non-resonance cond}%
\end{equation}

 $\bullet$ We require the following numerical non-resonance
 condition
\begin{equation}
|\sin(h\omega_j)|\geq h \epsilon^{1/2}\ \ for \ \ |j|\leq M.
 \label{further-non-res cond}
\end{equation}

 $\bullet$ For a positive constant $c > 0$, consider another non-resonance condition
\begin{equation}
\begin{aligned}
&|\sin(\frac{h}{2}(\omega_j-\mathbf{k}\cdot\boldsymbol{\omega}))\cdot
\sin(\frac{h}{2}(\omega_j+\mathbf{k}\cdot\boldsymbol{\omega}))| \geq
c h^2 |2\phi_2
(h^2\omega^2_j)|\\
&\textmd{for}\ (j, \mathbf{k}) \ \textmd{of the form}\ j=j_1+j_2\
\textmd{and}\ \mathbf{k}=\pm\langle j_1\rangle\pm\langle j_2\rangle,
 \end{aligned}
 \label{another-non-res cond}%
\end{equation}
 which leads to improved
conservation estimates.

\end{assum}

 We are now in a position to present the main result of this paper.
\begin{theo}\label{main theo} Define the following modified   momentum and
actions, respectively
\begin{equation*}
\begin{aligned}
&\hat{I}_j(\mathbf{q},\mathbf{p})= \frac{ \cos(\frac{1}{2}h
\omega_j)}{ \textmd{sinc}(\frac{1}{2}h
\omega_j)}I_j(\mathbf{q},\mathbf{p}) ,\ \
\hat{K}(\mathbf{q},\mathbf{p})=-\sum\limits_{|j|\leq
M}''\mathrm{i}j\frac{ \cos(\frac{1}{2}h \omega_j)}{
\textmd{sinc}(\frac{1}{2}h \omega_j)}q_{-j}p_{j}.
\end{aligned}
\end{equation*}
Suppose that the  conditions  of Assumptions \ref{ass} are true with
$s \geq \sigma + 1$. Then for the AAVF method \eqref{AAVFmethod} and
 $0\leq t=nh\leq \epsilon^{-N+1}$,
 the following near-conservation estimates of the modified momentum
and actions
\begin{equation*}
\begin{aligned}
&\sum\limits_{l=0}^{M}\omega_l^{2s+1}\frac{|\hat{I}_l(\mathbf{q}_{n},\mathbf{p}_{n})-\hat{I}_l(\mathbf{q}_0,\mathbf{p}_0)|}{\epsilon^2}
\leq
C  \epsilon,\\
&\frac{|\hat{K}(\mathbf{q}_{n},\mathbf{p}_{n})-\hat{K}(\mathbf{q}_0,\mathbf{p}_0)|}{\epsilon^2}
\leq C (\epsilon+M^{-s}+\epsilon t M^{-s+1})
\end{aligned}
\end{equation*}
hold with a  constant $C$  which depends on $s, N,$ and $C_0$, but
not on    $\epsilon, M, h$ and the time $t$. If
 \eqref{another-non-res cond} is not satisfied, then the
bound $C \epsilon$ is weakened to $C \epsilon^{1/2}$.
\end{theo}

The proof of this theorem will be  presented in detail in Section
\ref{sec:proof} by using the technique of multi-frequency modulated
Fourier expansions. It can be concluded from this theorem that the
AAVF method  has a near-conservation of   a modified momentum and
modified actions over long times.  Although the result cannot be
obtained for the momentum $K$ and actions $I_j$,  we note that  $K$
and   $I_j$ are  no longer exactly conserved quantities in the
semi-discretisation, which is seen from Theorem \ref{08 theo 1}.
Moreover, it will be shown in the next subsection that in comparison
with the near-conservation of  $K$ and  $I_j$, the modified momentum
and modified actions are preserved better by AAVF method. This
soundly supports the result of Theorem \ref{main theo}.

We have noticed that the authors in \cite{Cohen08-1} analysed the
long-time behaviour of a symmetric and symplectic trigonometric
integrator for solving  wave equations. It was shown in
\cite{Cohen08-1} that this integrator has a near-conservation   of
energy, momentum and actions in numerical discretisations. It is
noted that the method studied   in \cite{Cohen08-1} cannot preserve
the  energy \eqref{H} exactly. However, from Proposition \ref{prop
h} and Theorem \ref{main theo}, it follows that the AAVF method not
only preserves  the energy \eqref{H} exactly but also has a
near-conservation of modified momentum and actions over long times.

 \subsection{Numerical experiment} \begin{figure}[ptb]
\centering\tabcolsep=2mm
\begin{tabular}
[l]{lll}%
\includegraphics[width=12cm,height=4cm]{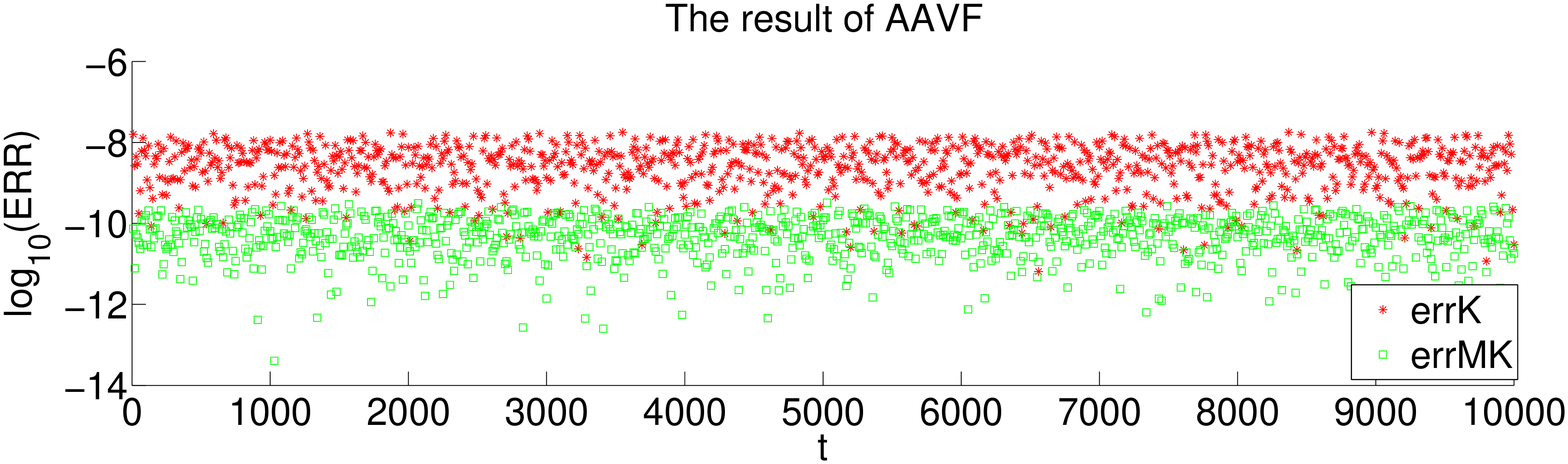}\\
\includegraphics[width=12cm,height=4cm]{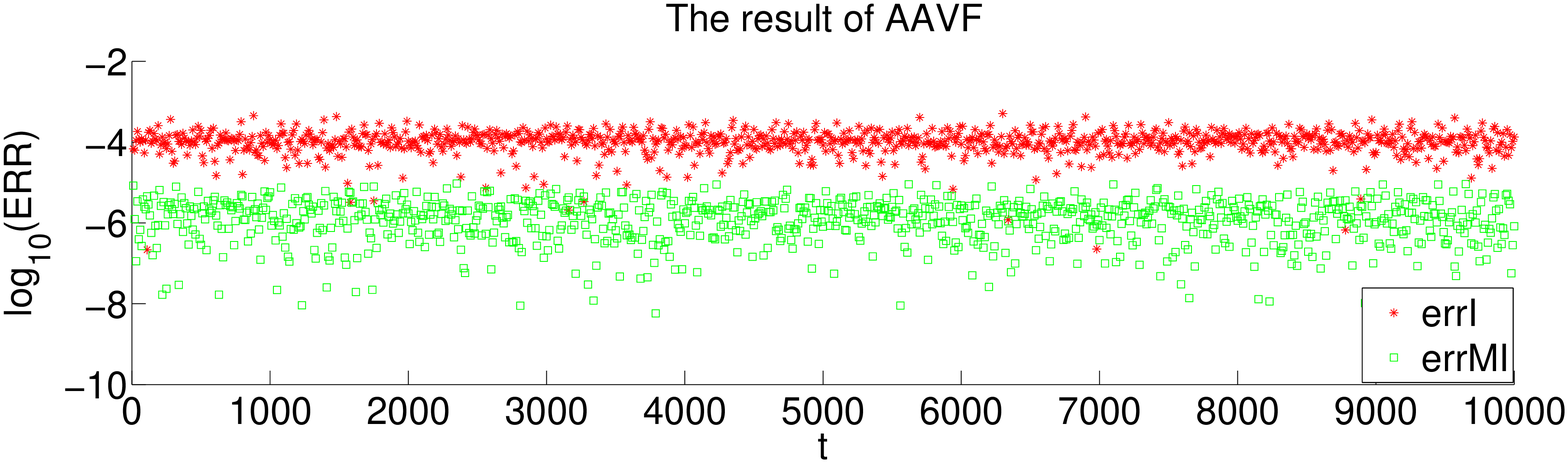}
\end{tabular}
\caption{The logarithm of the  errors against $t$.}%
\label{fig0}%
\end{figure}
We now carry out  a numerical experiment to show the numerical
behaviour of AAVF
 method. The semilinear wave equation \eqref{wave equa} with  $\rho=0.5$ and $g(u)=-u^2$ is considered (see \cite{Cohen08-1}) and its initial conditions
 are given by
$u(x,0)=0.1\big(\frac{x}{\pi}-1\big)^3\big(\frac{x}{\pi}+1\big)^2,\
\partial_t u(x,0)=0.01\frac{x}{\pi}\big(\frac{x}{\pi}-1\big)\big(\frac{x}{\pi}+1\big)^2$
for $-\pi\leq x\leq \pi$.
 We consider the spatial discretisation  with
the  dimension $2M = 2^7$ and consider    applying
 midpoint rule    to the integral appearing in   the AAVF formula
\eqref{AAVFmethod}. It can be checked that the assumption
\eqref{initi cond} holds for $s=2$. This problem is solved with  the
stepsize $h = 0.05$   on $[0,10000]$ and  the relative errors of
momentum/modified momentum and actions/modified actions against $t$
are shown in Figure \ref{fig0}.   Here we use the following
notations in the figures:
$\textmd{errK}=\frac{|K(\mathbf{q}_{n},\mathbf{p}_{n})-K(\mathbf{q}_0,\mathbf{p}_0)|}{|K(\mathbf{q}_0,\mathbf{p}_0)|},\
\textmd{errMK}=\frac{|\hat{K}(\mathbf{q}_{n},\mathbf{p}_{n})-\hat{K}(\mathbf{q}_0,\mathbf{p}_0)|}{|\hat{K}(\mathbf{q}_0,\mathbf{p}_0)|}$
and $\textmd{errI}=\frac{\sum\limits_{l=0}^{M}\omega_l^{5}
|I_l(\mathbf{q}_{n},\mathbf{p}_{n})-I_l(\mathbf{q}_0,\mathbf{p}_0)|}{\sum\limits_{l=0}^{M}\omega_l^{5}
|I_l(\mathbf{q}_0,\mathbf{p}_0)|}, \
\textmd{errMI}=\frac{\sum\limits_{l=0}^{M}\omega_l^{5}
|\hat{I}_l(\mathbf{q}_{n},\mathbf{p}_{n})-\hat{I}_l(\mathbf{q}_0,\mathbf{p}_0)|}{\sum\limits_{l=0}^{M}\omega_l^{5}
|\hat{I}_l(\mathbf{q}_0,\mathbf{p}_0)|}$. It follows from the
results that the modified momentum and modified actions are better
conserved than the momentum and actions, which supports the results
given in Theorem \ref{main theo}.

\section{The proof of the main result} \label{sec:proof}
\subsection{Preliminaries for the  analysis}
Define   five operators   by
\begin{equation*}
\begin{aligned}L_1^{\mathbf{k}}:&=\mathrm{e}^{\mathrm{i}(\mathbf{k} \cdot\boldsymbol{\omega})
h}\mathrm{e}^{\epsilon
hD}-2\cos(h\Omega)+\mathrm{e}^{-\mathrm{i}(\mathbf{k}
\cdot\boldsymbol{\omega})
h}\mathrm{e}^{-\epsilon hD},\\
L_2^{\mathbf{k}}:&=\mathrm{e}^{\frac{1}{2}\mathrm{i}(\mathbf{k}
\cdot\boldsymbol{\omega}) h}\mathrm{e}^{\frac{1}{2}\epsilon
hD}+\mathrm{e}^{-\frac{1}{2}\mathrm{i}(\mathbf{k}
\cdot\boldsymbol{\omega})
h}\mathrm{e}^{-\frac{1}{2}\epsilon hD},\\
L_3^{\mathbf{k}}:&=(\mathrm{e}^{\mathrm{i}(\mathbf{k}
\cdot\boldsymbol{\omega}) h}\mathrm{e}^{\epsilon
hD}-1)(\mathrm{e}^{\mathrm{i}(\mathbf{k} \cdot\boldsymbol{\omega})
h}\mathrm{e}^{\epsilon hD}+1)^{-1},\\
L_4^{\mathbf{k}}(\sigma):&=
(1-\sigma)\mathrm{e}^{-\frac{1}{2}\mathrm{i}(\mathbf{k}
\cdot\boldsymbol{\omega}) h}\mathrm{e}^{-
 \frac{h}{2}\epsilon D} +\sigma\mathrm{e}^{\frac{1}{2}\mathrm{i}(\mathbf{k} \cdot\boldsymbol{\omega})
h}\mathrm{e}^{
 \frac{h}{2}\epsilon D},\\
 L^{\mathbf{k}}:&=(L^{\mathbf{k}}_2)^{-1} L^{\mathbf{k}}_1,\end{aligned}
\end{equation*} where $D$ is the
differential operator (see \cite{hairer2006}). For these operators,
the following results are clear.
\begin{prop}\label{lhd pro}
The operator $L^{\mathbf{k}}$ can be expressed in   Taylor
expansions as follows:
\begin{equation}\label{L expansion}
\begin{aligned}&L^{\pm\langle j\rangle}(hD)\alpha_j^{\pm\langle j\rangle}(\epsilon t)= \pm2 \textmd{i}
\epsilon h  s_{\langle j\rangle} \dot{\alpha}_j^{\pm\langle
j\rangle}(\epsilon t)+\frac{1}{2} \epsilon^2
h^2\sec(\frac{1}{2}h\omega_j)\ddot{\alpha}_j^{\pm\langle
j\rangle}(\epsilon t)+\cdots,\\
 &L^{\mathbf{k}}(hD)\alpha_j^{\mathbf{k}}(\epsilon t)=  2 \frac{s_{\langle j\rangle+\mathbf{k}}s_{\langle
j\rangle-k}}{c_{\mathbf{k}}}\alpha_j^{\mathbf{k}}(\epsilon
t)+\textmd{i} \epsilon h
 \frac{s_{\mathbf{k}}(1+c_{\langle j\rangle+\mathbf{k}}c_{\langle
j\rangle-\mathbf{k}})}{c_{\mathbf{k}}^2}\dot{\alpha}_j^{\mathbf{k}}(\epsilon t)+\cdots,\\
\end{aligned}
\end{equation}
for $\abs{j}>0$ and $\mathbf{k}\neq \pm\langle j\rangle$, where
$s_{\mathbf{k}} =\sin(
\frac{h}{2}(\mathbf{k}\cdot\boldsymbol{\omega}))$ and
$c_{\mathbf{k}} =\cos(
\frac{h}{2}(\mathbf{k}\cdot\boldsymbol{\omega}))$. The Taylor
expansions of $ L_3^{\mathbf{k}}$ are given by
\begin{equation*}
\begin{aligned}
 &L_3^{\mathbf{k}}\alpha_j^{\mathbf{k}}(\epsilon t)=   \textmd{i}\tan\big(\frac{1}{2}h(\mathbf{k} \cdot\boldsymbol{\omega})\big)\alpha_j^{\mathbf{k}}(\epsilon t)+
\frac{h\epsilon}{1+c_{2 k}} \dot{\alpha}_j^{\mathbf{k}}(\epsilon t)+\cdots,\\
\end{aligned}
\end{equation*}
for $\abs{j}>0$ and $\norm{\mathbf{k}}\leq 2N$. Moreover, for the
operator $L_4^{\mathbf{k}}(\sigma)$ with $\norm{\mathbf{k}}\leq 2N$,
we have
\begin{equation*}
\begin{aligned}
L_4^{\mathbf{k}}(\frac{1}{2})=\cos\big(\frac{h(\mathbf{k} \cdot
\boldsymbol{\omega})}{2}\big)+\frac{1}{2}\sin\big(\frac{h(\mathbf{k}
\cdot \boldsymbol{\omega})}{2}\big)(\textmd{i} h\epsilon D)+\cdots.
\end{aligned}
\end{equation*}
\end{prop}
The following lemma is given in \cite{Cohen08} which will be
 needed in the analysis of this paper.
\begin{lem}\label{lem 08} (See \cite{Cohen08}.)
For $s>1/2$, one has $ \sum\limits_{\norm{\mathbf{k}}\leq K}
\boldsymbol{\omega}^{-2s\abs{\mathbf{k}}}\leq C_{K,s}\leq\infty. $
For $s>1/2$ and $m\geq 2$, it is true that
\begin{equation*}
  \sup_{\norm{\mathbf{k}}\leq K}
\sum\limits_{\mathbf{k}^1+\cdots+\mathbf{k}^m=\mathbf{k}}
\frac{\boldsymbol{\omega}^{-2s(|\mathbf{k}^1|+\cdots+|\mathbf{k}^m|)}}{\boldsymbol{\omega}^{-2s|\mathbf{k}|}}
\leq C_{m,K,s}<\infty,
\end{equation*}
where the sum is taken over $(\mathbf{k}^1,\ldots,\mathbf{k}^m)$
satisfying $\norm{\mathbf{k}^i}\leq K$. For $s\geq1$, it is further
true that $
  \sup_{\norm{\mathbf{k}}\leq K}
\frac{\sum\limits_{l\geq0}\abs{k_l}\omega_l^{2s+1}}{\boldsymbol{\omega}^{2s|\mathbf{k}|}(1+\abs{\mathbf{k}
\cdot\boldsymbol{\omega}})}
  \leq C_{K,s}<\infty.
$
\end{lem}

\subsection{The outline of the proof}
The proof relies on a careful research of a modulated Fourier
expansion of the AAVF method \eqref{AAVFmethod}.
 Assume that the conditions of Theorem \ref{main theo} are true. For
the numerical solution $(\mathbf{q}_{n}, \mathbf{p}_{n})$ given by
\eqref{AAVFmethod}, we will construct the following truncated
multi-frequency modulated Fourier expansion (with $N$ from
\eqref{near-resonant R})
\begin{equation}
\begin{aligned} &\tilde{\mathbf{q}}(t)=
\sum\limits_{\norm{\mathbf{k}}\leq 2N}
\mathrm{e}^{\mathrm{i}(\mathbf{k} \cdot\boldsymbol{\omega})
t}\zeta^{\mathbf{k}}(\epsilon t),\ \ \tilde{\mathbf{p}}(t)=
\sum\limits_{\norm{\mathbf{k}}\leq2N}
\mathrm{e}^{\mathrm{i}(\mathbf{k} \cdot\boldsymbol{\omega})
t}\eta^{\mathbf{k}}(\epsilon t),
\end{aligned}
\label{MFE-AAVF}%
\end{equation}
where  $t=nh$ and
$\zeta_{-j}^{-\mathbf{k}}=\overline{\zeta_{j}^{\mathbf{k}}}$,
$\eta_{-j}^{-\mathbf{k}}=\overline{\eta_{j}^{\mathbf{k}}}$. For this
modulated Fourier expansion,   the following key points will be
considered one by one in the rest of this section.
  \begin{itemize}
 \item In Sect. \ref{subsec:mod equ}   formal modulation equations  for the modulation
functions are derived.

\item In Sect. \ref{subsec:rev pic} we consider an iterative construction of the
functions using  reverse Picard iteration.

\item We then  work with a more convenient rescaling and study the estimation of  non-linear terms  in Sect. \ref{subsec:res
est}.

\item  Abstract reformulation of the iteration is presented in Sect. \ref{subsec:ref
rev}.

\item In Sect. \ref{subsec:bou coe}  we control the size of the numerical solution by   studying the bounds of modulation functions.

\item In Sect. \ref{subsec:def}  the bound of the defect is  estimated.

\item We study the
difference of the numerical solution and its modulated Fourier
expansion in Sect.  \ref{subsec:rem}.

\item In Sect. \ref{subsec:alm inv} we show two invariants of the
modulation system  and establish their relationship with the
modified momentum and modified actions.

\item Finally, the previous results that are valid only on a short time
interval  are extended  to a long time interval in Sect.
\ref{subsec:fro to}.
\end{itemize}

 It is noted that the above  procedure is  a standard approach to studying the
long-time behavior of numerical methods of Hamiltonian partial
differential equations by using modulated Fourier expansions (see,
e.g. \cite{Cohen08,Cohen08-1,Gauckler16,Gauckler12,Hairer08}). The
proof  presented here closely follows   these previous publications
but with some modifications adapted to the AAVF method. The main
differences in the analysis   arise due to the implicitness of the
AAVF method and the integral appearing in the method.
\subsection{Modulation equations}\label{subsec:mod equ}
Throughout the proof,   denote by $C$ a generic constant which is
independent of $\epsilon, M, h$ and   $t=nh$.

 In the light of the symmetry of the
AAVF
 method and the following property
$$\displaystyle\int_{0}^{1}f((1-\sigma)\mathbf{q}_{n}+\sigma
\mathbf{q}_{n-1})d\sigma=\displaystyle\int_{0}^{1}f((1-\sigma)\mathbf{q}_{n-1}+\sigma
\mathbf{q}_{n})d\sigma,$$ one obtains
\begin{equation}
\begin{aligned}&\mathbf{q}_{n+1}-2\cos(h\Omega)\mathbf{q}_{n}+\mathbf{q}_{n-1}\\
=&h^2\phi_2(V)\big[\displaystyle\int_{0}^{1}f((1-\sigma)\mathbf{q}_{n}+\sigma
\mathbf{q}_{n+1})d\sigma+
\displaystyle\int_{0}^{1}f((1-\sigma)\mathbf{q}_{n-1}+\sigma
\mathbf{q}_{n})d\sigma\big].
\end{aligned}\label{MFE-2}%
\end{equation}

 We look for a modulated Fourier expansion of the form
\begin{equation*}
\begin{aligned} &\tilde{\mathbf{q}}_{h}(t+\frac{h}{2},\sigma)=\sum\limits_{\norm{\mathbf{k}}\leq 2N}
\mathrm{e}^{\mathrm{i}(\mathbf{k}\cdot\boldsymbol{\omega})
(t+\frac{h}{2})}\xi^{\mathbf{k}}\Big(\epsilon(t+\frac{h}{2}),\sigma\Big)
\end{aligned}
\end{equation*}
for the   term $(1-\sigma)\mathbf{q}_{n}+\sigma \mathbf{q}_{n+1}$.
Then it is obtained that
\begin{equation}\label{xip}
\begin{aligned} \xi^{\mathbf{k}}\Big(\epsilon(t+\frac{h}{2}),\sigma\Big)=&\Big((1-\sigma)\mathrm{e}^{-\frac{1}{2}\mathrm{i}(\mathbf{k} \cdot
\boldsymbol{\omega}) h}\mathrm{e}^{-
 \frac{h}{2}\epsilon D} +\sigma\mathrm{e}^{\frac{1}{2}\mathrm{i}(\mathbf{k} \cdot\boldsymbol{\omega})
h}\mathrm{e}^{
 \frac{h}{2}\epsilon D}\Big)\zeta^{\mathbf{k}}\Big(\epsilon(t+\frac{h}{2})\Big)\\
 =&L^{\mathbf{k}}_4(\sigma)\zeta^{\mathbf{k}}\Big(\epsilon(t+\frac{h}{2})\Big).
\end{aligned}
\end{equation}
In the same way, for $(1-\sigma)\mathbf{q}_{n-1}+\sigma
\mathbf{q}_{n}$, we have the following modulated Fourier expansion
\begin{equation*}
\begin{aligned} &\tilde{\mathbf{q}}_{h}(t+\frac{h}{2},\sigma)=\sum\limits_{\norm{\mathbf{k}}\leq 2N}
\mathrm{e}^{\mathrm{i}(\mathbf{k}\cdot\boldsymbol{\omega})
(t+\frac{h}{2})}\xi^{\mathbf{k}}\Big(\epsilon(t-\frac{h}{2}),\sigma\Big)
\end{aligned}
\end{equation*}
with
\begin{equation}\label{xim}
\begin{aligned} \xi^{\mathbf{k}}\Big(\epsilon(t-\frac{h}{2}),\sigma\Big)=L^{\mathbf{k}}_4(\sigma)\zeta^{\mathbf{k}}\big(\epsilon(t-\frac{h}{2})\big).
\end{aligned}
\end{equation}
Inserting the  modulated Fourier expansions \eqref{MFE-AAVF},
\eqref{xip}, and \eqref{xim} into \eqref{MFE-2} yields
\begin{equation*}
\begin{aligned}&\tilde{\mathbf{q}}(t+h)-2\cos(h\Omega)\tilde{\mathbf{q}}(t)+\tilde{\mathbf{q}}(t-h)\\
=&h^2\phi_2(V)\Big[\displaystyle\int_{0}^{1}f\big(\tilde{\mathbf{q}}_{h}(t+\frac{h}{2},\sigma)\big)d\sigma+
\displaystyle\int_{0}^{1}f\big(\tilde{\mathbf{q}}_{h}(t-\frac{h}{2},\sigma)\big)d\sigma\Big],
\end{aligned}
\end{equation*}
which can be rewritten  as \begin{equation}\label{new
q}(\mathrm{e}^{\frac{1}{2} hD}+\mathrm{e}^{-\frac{1}{2} hD})^{-1}(
\mathrm{e}^{ hD}-2\cos(h\Omega)+ \mathrm{e}^{-
hD})\tilde{\mathbf{q}}(t)=h^2\phi_2(V)\displaystyle\int_{0}^{1}f(\tilde{\mathbf{q}}_{h}(t,\sigma))d\sigma.\end{equation}

In what follows, we rewrite this equation by using the same way
introduced in  \cite{Hairer08}. We start with making the following
notation.  For a $2\pi$-periodic function $w(x)$, denote by
$(\mathcal{Q}w)(x)$ the trigonometric interpolation polynomial to
$w(x)$ in the points $x_k$. If $w(x)$ is of the form
$w(x)=\sum\limits_{j=-\infty}^{\infty} w_j e^{\mathrm{i}jx},$ then
one has that $(\mathcal{Q}w)(x)= \sum\limits_{\abs{j} \leq
M}''(\sum\limits_{l=-\infty}^{\infty}w_{j+2Ml}) e^{\mathrm{i}jx}$ by
considering $x_k=\frac{k \pi}{M}.$ For a $2M$-periodic coefficient
sequence $\mathbf{q} = (q_j)$,  $(\mathcal{P}\mathbf{q})(x)$ is
referred to the trigonometric polynomial with coefficients $q_j$,
i.e., $(\mathcal{P}\mathbf{q})(x) = \sum\limits_{\abs{j} \leq M}'
q_j e^{\mathrm{i}jx}.$ By using these new denotations,  \eqref{new
q} becomes
\begin{equation}\label{new qq}
(\mathrm{e}^{\frac{1}{2} hD}+\mathrm{e}^{-\frac{1}{2} hD})^{-1}(
\mathrm{e}^{ hD}-2\cos(h\Omega)+ \mathrm{e}^{-
hD})\mathcal{P}\tilde{\mathbf{q}}(t)=h^2\phi_2(V)\displaystyle\int_{0}^{1}\mathcal{Q}g(\mathcal{P}\tilde{\mathbf{q}}_{h}(t,\sigma))d\sigma.\end{equation}
Taylor expansion of the non-linearity $\mathcal{Q}g$ at $0$ is given
by
\footnote{It is noted that $g(0) = g'(0) = 0$ is used here.}
\begin{equation*}
\begin{aligned}&\mathcal{Q}g(\mathcal{P}\tilde{\mathbf{q}}_{h}(t,\sigma))=\sum\limits_{m\geq
2}\frac{g^{(m)}(0)}{m!}\mathcal{Q}(\mathcal{P}\tilde{\mathbf{q}}_{h}(t,\sigma))^m\\
=&\sum\limits_{m\geq 2}\frac{g^{(m)}(0)}{m!}
\Big(\sum\limits_{\abs{j_1} \leq M}''
\sum\limits_{l=-\infty}^{\infty}
\sum\limits_{\norm{\mathbf{k}^1}\leq 2N}'
\mathrm{e}^{\mathrm{i}(\mathbf{k}^1\cdot\boldsymbol{\omega})
t}\xi^{\mathbf{k}^1}_{j_1+2Ml}(\tau,\sigma)
e^{\mathrm{i}j_1x}\Big)\\
&\cdots \Big(\sum\limits_{\abs{j_m} \leq
M}''\sum\limits_{l=-\infty}^{\infty}
\sum\limits_{\norm{\mathbf{k}^m}\leq 2N}'
\mathrm{e}^{\mathrm{i}(\mathbf{k}^m\cdot\boldsymbol{\omega})
t}\xi^{\mathbf{k}^m}_{j_m+2Ml}(\tau,\sigma) e^{\mathrm{i}j_mx}\Big)\\
=&\sum\limits_{m\geq 2}\frac{g^{(m)}(0)}{m!} \sum\limits_{\abs{j}
\leq M}''\sum\limits_{j_1+\cdots+j_m\equiv j\ \textmd{mod}\ 2M}'
\sum\limits_{\norm{\mathbf{k}^1}\leq
2N,\ldots,\norm{\mathbf{k}^m}\leq 2N}
(\xi^{\mathbf{k}^1}_{j_1}\cdots \xi^{\mathbf{k}^m}_{j_m})(\tau,\sigma)\\
&
\mathrm{e}^{\mathrm{i}((\mathbf{k}^1+\cdots+\mathbf{k}^m)\cdot\boldsymbol{\omega})
t} e^{\mathrm{i}jx},
\end{aligned}
\end{equation*}
where $\tau=h \epsilon$ and the prime on the sum indicates that a
factor $1/2$ is included in the appearance of
$\xi^{\mathbf{k}^i}_{j_i}$ with $j_i = \pm M$.
  Inserting this into \eqref{new qq},
considering the $j$th Fourier coefficient and comparing the
coefficients of
$\mathrm{e}^{\mathrm{i}(\mathbf{k}\cdot\boldsymbol{\omega}) t}$, we
obtain
\begin{equation}\label{ljkqp}
\begin{aligned}L^{\mathbf{k}} \zeta_j^{\mathbf{k}}=&-h^2\phi_2(h^2\omega^2_j)\sum\limits_{m\geq 2}\frac{g^{(m)}(0)}{m!}
\sum\limits_{\mathbf{k}^1+\cdots+\mathbf{k}^m=\mathbf{k}}\sum\limits_{j_1+\cdots+j_m\equiv
j\
\textmd{mod}\ 2M}'\\
& \int_{0}^{1}\Big[\big(\xi_{j_1}^{\mathbf{k}^1}\cdot\ldots\cdot
\xi_{j_m}^{\mathbf{k}^m}\big)(t\epsilon,\sigma)\Big]d\sigma.
\end{aligned}
\end{equation}
 It is noted that the integral
appearing here can be calculated exactly.

According to
 the Taylor expansion \eqref{L expansion} of $L^{\mathbf{k}}$, the
dominating term   is $\pm2\mathrm{i}\sin(\frac{1}{2}h\omega_j)
h\epsilon \dot{\zeta}_j^{\pm \langle j\rangle}$  for  $\mathbf{k} =
\pm \langle j\rangle$. If $\mathbf{k}\neq \pm \langle j\rangle$,
then the dominating term is $2 \frac{s_{\langle
j\rangle+\mathbf{k}}s_{\langle
j\rangle-\mathbf{k}}}{c_{\mathbf{k}}}$  by considering the condition
\eqref{inequa}. If
 \eqref{inequa} is not true, the  condition
\eqref{non-resonance cond} ensures that the defect in simply setting
$ \zeta_j^{\mathbf{k}}\equiv 0$ is of size
$\mathcal{O}(\epsilon^{N+1})$ in an appropriate Sobolev-type norm.
The above analysis and \eqref{ljkqp} determine the formal modulation
equations of  modulated functions $\zeta^{\mathbf{k}}$.

For the  modulation equations of $\eta^{\mathbf{k}}$,    it follows
from \eqref{AAVFmethod} that  \begin{equation}\label{qp connec}
\mathbf{q}_{n+1}-\mathbf{q}_{n} =\Omega^{-1}\tan(\frac{1}{2}h
\Omega) ( \mathbf{p}_{n+1}+ \mathbf{p}_{n} ).\end{equation} By the
definition of $L_3$, this relation can be expressed as $$
L_3^{\mathbf{k}}\zeta^{\mathbf{k}}=\Omega^{-1}\tan(\frac{1}{2}h
\Omega) \eta^{\mathbf{k}}.$$ In terms of the Taylor series of
$L^{\mathbf{k}}_3$, the relationship between $\eta^{\mathbf{k}}$ and
$\zeta^{\mathbf{k}}$ can be established:
\begin{equation}\label{eta modula sys}
\begin{aligned} & \eta_j^{\pm\langle j\rangle}=\pm \mathrm{i}\omega_j \zeta_j^{\pm\langle
j\rangle}+\mathcal{O}(h \epsilon),\ \   \eta_j^{\mathbf{k}}=
\frac{\tan(\frac{1}{2}h(\mathbf{k}\cdot
\boldsymbol{\omega}))}{\tan(\frac{1}{2}h\omega_j)}\mathrm{i}\omega_j\zeta_j^{\mathbf{k}}+\mathcal{O}(h
\epsilon)
\end{aligned}
\end{equation}
for $\mathbf{k}\neq \pm\langle j\rangle,$ which gives the modulation
equations of $\eta^{\mathbf{k}}$.

On the other hand,   it needs to derive the initial values
 for  $\dot{\zeta}_j^{\pm \langle j\rangle}$ appearing in \eqref{ljkqp}. By considering $\tilde{\mathbf{q}} (0) =
\mathbf{q}(0)$, one has
\begin{equation}\label{initial pl}
\begin{aligned}&\zeta_j^{\langle
j\rangle}(0)+\zeta_j^{-\langle
j\rangle}(0)=q_j(0)-\sum\limits_{\mathbf{k}\neq \pm  \langle
j\rangle} \zeta_j^{\mathbf{k}}(0).
\end{aligned}
\end{equation}
Moreover, it follows from  $\tilde{\mathbf{p}} (0) = \mathbf{p}(0)$
that $\eta_j^{\langle j\rangle}(0)+\eta_j^{-\langle
j\rangle}(0)=p_j(0)-\sum\limits_{\mathbf{k}\neq \pm  \langle
j\rangle} \eta_j^{\mathbf{k}}(0),$  which yields
\begin{equation}\label{initial mi}
\begin{aligned}&\mathrm{i} \omega_j(\zeta_j^{\langle
j\rangle}(0)-\zeta_j^{-\langle
j\rangle}(0))=p_j(0)-\sum\limits_{\mathbf{k}\neq \pm  \langle
j\rangle}
\eta_j^{\mathbf{k}}(0)\\
&=p_j(0)-\sum\limits_{\mathbf{k}\neq \pm  \langle
j\rangle}\frac{\tan(\frac{1}{2}h(\mathbf{k}\cdot\boldsymbol{\omega}))}{\tan(\frac{1}{2}h
\omega_j)}\mathrm{i}\omega_j
\zeta_j^{\mathbf{k}}(0)+\mathcal{O}(h\epsilon).
\end{aligned}
\end{equation}
The formulae \eqref{initial pl} and \eqref{initial mi} determine the
initial values for $\zeta_j^{\pm \langle j\rangle}$.

\subsection{Reverse Picard iteration}\label{subsec:rev pic}
  Following
\cite{Cohen08-1,Hairer08}, the reverse Picard iteration of the
functions $\zeta^{\mathbf{k}}$  is considered here such that after
$4N$ iteration steps, the defects in \eqref{ljkqp}, \eqref{initial
pl} and \eqref{initial mi} are of magnitude
$\mathcal{O}(\epsilon^{N+1})$ in the $H^s$ norm.

Denote  by $[\cdot]^{(n)}$ the $n$th iterate.    For $\mathbf{k} =
\pm \langle j\rangle$, we consider the iteration procedure as
follows:
\begin{equation}\label{Pic ite j}
\begin{aligned}&\pm2\mathrm{i}s_{\langle
j\rangle}h \epsilon \big[\dot{\zeta}_j^{\pm  \langle
j\rangle}\big]^{(n+1)}=\Big[-h^2\phi_2(h^2\omega^2_j)\sum\limits_{m\geq
2}\frac{g^{(m)}(0)}{m!}
\sum\limits_{\mathbf{k}^1+\cdots+\mathbf{k}^m=\mathbf{k}}\sum\limits_{j_1+\cdots+j_m\equiv
j\
\textmd{mod}\ 2M}'\\
&  \int_{0}^{1}\Big[\big(\xi_{j_1}^{\mathbf{k}^1}\cdot\ldots\cdot
\xi_{j_m}^{\mathbf{k}^m}\big)(t\epsilon,\sigma)\Big]d\sigma-\big(\frac{1}{2}
\epsilon^2 h^2\sec(\frac{1}{2}h\omega_j) \ddot{\zeta}_j^{\pm \langle
j\rangle}+\cdots\big)\Big]^{(n)}.
\end{aligned}
\end{equation}  For $\mathbf{k} \neq\pm  \langle j\rangle$
 and
$j$ satisfying the non-resonant \eqref{inequa}, the  iteration
procedure is of the form
\begin{equation}\label{Pic ite notj}
\begin{aligned}&2 \frac{s_{\langle j\rangle+\mathbf{k}}s_{\langle
j\rangle-\mathbf{k}}}{c_{\mathbf{k}}}
\big[\zeta_j^{\mathbf{k}}\big]^{(n+1)}=\Big[-h^2\phi_2(h^2\omega^2_j)\sum\limits_{m\geq
2}\frac{g^{(m)}(0)}{m!}
\sum\limits_{\mathbf{k}^1+\cdots+\mathbf{k}^m=\mathbf{k}}\sum\limits_{j_1+\cdots+j_m\equiv
j\
\textmd{mod}\ 2M}'\\
&  \int_{0}^{1}\Big[\big(\xi_{j_1}^{\mathbf{k}^1}\cdot\ldots\cdot
\xi_{j_m}^{\mathbf{k}^m}\big)(t\epsilon,\sigma)\Big]d\sigma-\big(\textmd{i}
\epsilon h
 \frac{s_{\mathbf{k}}(1+c_{\langle j\rangle+\mathbf{k}}c_{\langle
j\rangle-\mathbf{k}})}{c_{\mathbf{k}}^2}
\dot{\zeta}_j^{\mathbf{k}}+\cdots\big)\Big]^{(n)},
\end{aligned}
\end{equation}
where we let $\zeta_j^{\mathbf{k}}=0$ for  $\mathbf{k}\neq
\pm\langle j\rangle$ in the near-resonant set
$\mathcal{R}_{\epsilon,h}$.  For the initial values \eqref{initial
pl} and \eqref{initial mi}, the iteration procedure  reads
\begin{equation}\label{pica initial pl}
\begin{aligned}&\big[\zeta_j^{\langle
j\rangle}(0)+\zeta_j^{-\langle
j\rangle}(0)\big]^{(n+1)}=\big[q_j(0)-\sum\limits_{\mathbf{k}\neq
\pm \langle j\rangle} \zeta_j^{\mathbf{k}}(0)\big]^{(n)},\\
&\mathrm{i} \omega_j\big[\zeta_j^{\langle
j\rangle}(0)-\zeta_j^{-\langle
j\rangle}(0)\big]^{(n+1)}=\big[p_j(0)-\sum\limits_{\mathbf{k}\neq
\pm \langle j\rangle}\frac{\tan(\frac{1}{2}h(\mathbf{k}\cdot
\boldsymbol{\omega}))}{\tan(\frac{1}{2}h
\omega_j)}\mathrm{i}\omega_j
\zeta_j^{\mathbf{k}}(0)+\mathcal{O}(h\epsilon)\big]^{(n)}.
\end{aligned}
\end{equation}

In these iterations it is    assumed that $\norm{\mathbf{k}}\leq K:=
2N$ and $\norm{\mathbf{k}^i}\leq K$ for $i = 1,\ldots ,m$. There is
an initial value problem of first-order ODEs for $\zeta_j^{\pm
\langle j\rangle}$ (for $|j|\leq M$) and algebraic equations for
$\zeta_j^{\mathbf{k}}$   with $\mathbf{k}\neq\pm \langle j\rangle$
at each iteration step. The starting iterates ($n = 0$) are chosen
as $\zeta_j^{\mathbf{k}}(\tau)=0$   for $\mathbf{k}\neq\pm \langle
j\rangle$, and $\zeta_j^{\pm\langle
j\rangle}(\tau)=\zeta_j^{\pm\langle j\rangle}(0)$, where
$\zeta_j^{\pm\langle j\rangle}(0)$   are
determined by \eqref{pica initial pl}. 

\subsection{Rescaling and estimation of the nonlinear terms}\label{subsec:res est}
Similarly to  Sect. 3.5 of \cite{Cohen08} and Sect. 6.3 of
\cite{Cohen08-1}, in what follows, we consider  a more convenient
rescaling
\begin{equation*}
\begin{aligned}
c\zeta_{j}^{\mathbf{k}}=\frac{\boldsymbol{\omega}^{|\mathbf{k}|}}{\epsilon^{[[\mathbf{k}]]}}\zeta_j^{\mathbf{k}},\
\ c\zeta^{\mathbf{k}}=\big(c\zeta_{j}^{\mathbf{k}}\big)_{|j|\leq
M}=\frac{\boldsymbol{\omega}^{|\mathbf{k}|}}{\epsilon^{[[\mathbf{k}]]}}\zeta^{\mathbf{k}}
\end{aligned}
\end{equation*}
 in the space $\mathrm{H}^s = (H^s)^{\mathcal{K}} = \{c\zeta =
(c\zeta^{\mathbf{k}})_{\mathbf{k}\in \mathcal{K}}:
c\zeta^{\mathbf{k}}\in H^s \}.$ The norm of this space is defined as
$|||c\zeta|||_s^2=\sum\limits_{\mathbf{k}\in \mathcal{K}}
\norm{c\zeta^{\mathbf{k}}}_s^2$, where  the set $\mathcal{K}$ is
defined by $\mathcal{K}=\{\mathbf{k}=(k_l)_{l=0}^M\ \textmd{with
integers} \ k_l:\ \norm{\mathbf{k}}\leq K\}$ with $K = 2N.$
Likewise, we use the notation $c\xi^{\mathbf{k}}\in \mathrm{H}^s$
with the same meaning.

In order to express the non-linearity of \eqref{ljkqp} in these
rescaled variables, define the nonlinear function
$\textbf{f}=(f_j^{\mathbf{k}})$ by
\begin{equation*}
\begin{aligned}f_j^{\mathbf{k}}\big(c\xi(\tau)\big) =&\frac{\boldsymbol{\omega}^{|\mathbf{k}|}}{\epsilon^{[[\mathbf{k}]]}}\sum\limits_{m=2}^N\frac{g^{(m)}(0)}{m!}
\sum\limits_{\mathbf{k}^1+\cdots+\mathbf{k}^m=\mathbf{k}}\frac{\epsilon^{[[\mathbf{k}^1]]+\cdots+[[\mathbf{k}^m]]}}{\boldsymbol{\omega}^{|\mathbf{k}^1|+\cdots+|\mathbf{k}^m|}} \\
&\sum\limits_{j_1+\cdots+j_m\equiv j\ \textmd{mod}\ 2M}'
\int_{0}^{1}\big(c\xi_{j_1}^{\mathbf{k}^1}\cdot\ldots\cdot
c\xi_{j_m}^{\mathbf{k}^m}\big)(\tau,\sigma)d\sigma.
\end{aligned}
\end{equation*}
Regarding  this function, we have the following bounds.
\begin{prop}\label{pro: f}
It is true that
\begin{equation}\label{bounds f12}
\begin{aligned}
\sum\limits_{\mathbf{k}\in \mathcal{K}}
\norm{f^{\mathbf{k}}(c\xi)}_s^2  \leq C \epsilon P(||| c\tilde{\xi}
|||_s^2),\ \
 \sum\limits_{|j|\leq M}
\norm{f^{\pm\langle j\rangle }(c\xi)}_s^2\leq C\epsilon^3
P_1(|||c\tilde{\xi}|||_s^2),
\end{aligned}
\end{equation}
where  $c\tilde{\xi}(\tau):=\sup_{0\leq
\sigma\leq1}\{c\xi(\tau,\sigma)\}$ and $P$ and $P_1$ are polynomials
with coefficients bounded independently of $\epsilon, h,$ and $M$.
\end{prop}
\begin{proof}
In the light of the inequality $(\sum\limits_{l=1}^N a_m)^2\leq
N\sum\limits_{l=1}^N a_m^2,$ and  the Cauchy-Schwarz inequality, one
has
\begin{equation*}
\begin{aligned}&\sum\limits_{\mathbf{k}\in \mathcal{K}}
\norm{f^{\mathbf{k}}(c\xi)}_s^2=\sum\limits_{\norm{\mathbf{k}}\leq
K}
\sum\limits_{|j|\leq M}''\omega_j^{2s} |f^{\mathbf{k}}_j|^2\\
=&\sum\limits_{\norm{\mathbf{k}}\leq K} \sum\limits_{|j|\leq
M}''\omega_j^{2s}\frac{\boldsymbol{\omega}^{2|\mathbf{k}|}}{\epsilon^{2[[\mathbf{k}]]}}\Big(
\sum\limits_{m=2}^N\frac{g^{(m)}(0)}{m!}
\sum\limits_{\mathbf{k}^1+\cdots+\mathbf{k}^m=\mathbf{k}}\frac{\epsilon^{[[\mathbf{k}^1]]+\cdots+[[\mathbf{k}^m]]}}{\boldsymbol{\omega}^{|\mathbf{k}^1|+\cdots+|\mathbf{k}^m|}} \\
&\sum\limits_{j_1+\cdots+j_m\equiv j\ \textmd{mod}\ 2M}'
\int_{0}^{1}\big(c\xi_{j_1}^{\mathbf{k}^1}\cdots
c\xi_{j_m}^{\mathbf{k}^m}\big)(\tau,\sigma)d\sigma\Big)^2\\
\leq&\sum\limits_{\norm{\mathbf{k}}\leq K} N\sum\limits_{|j|\leq
M}''\omega_j^{2s}\frac{\boldsymbol{\omega}^{2|\mathbf{k}|}}{\epsilon^{2[[\mathbf{k}]]}}
\sum\limits_{m=2}^N\Big(\frac{g^{(m)}(0)}{m!}\Big)^2
\Big(\sum\limits_{\mathbf{k}^1+\cdots+\mathbf{k}^m=\mathbf{k}}\frac{\epsilon^{[[\mathbf{k}^1]]+\cdots+[[\mathbf{k}^m]]}}{\boldsymbol{\omega}^{|\mathbf{k}^1|+\cdots+|\mathbf{k}^m|}} \\
&\sum\limits_{j_1+\cdots+j_m\equiv j\ \textmd{mod}\ 2M}'
\int_{0}^{1}\big(c\xi_{j_1}^{\mathbf{k}^1}\cdots
c\xi_{j_m}^{\mathbf{k}^m}\big)(\tau,\sigma)d\sigma\Big)^2\\
\leq&\sum\limits_{\norm{\mathbf{k}}\leq K} N\sum\limits_{|j|\leq
M}''\omega_j^{2s}\frac{\boldsymbol{\omega}^{2|\mathbf{k}|}}{\epsilon^{2[[\mathbf{k}]]}}
\sum\limits_{m=2}^N\Big(\frac{g^{(m)}(0)}{m!}\Big)^2
\sum\limits_{\mathbf{k}^1+\cdots+\mathbf{k}^m=\mathbf{k}}\Big(\frac{\epsilon^{[[\mathbf{k}^1]]+\cdots+[[\mathbf{k}^m]]}}{\boldsymbol{\omega}^{|\mathbf{k}^1|+\cdots+|\mathbf{k}^m|}} \Big)^2\\
&\sum\limits_{\mathbf{k}^1+\cdots+\mathbf{k}^m=\mathbf{k}}\Big(\sum\limits_{j_1+\cdots+j_m\equiv
j\ \textmd{mod}\ 2M}'
\int_{0}^{1}\big(c\xi_{j_1}^{\mathbf{k}^1}\cdots
c\xi_{j_m}^{\mathbf{k}^m}\big)(\tau,\sigma)d\sigma\Big)^2\\
\\
\leq&\sum\limits_{\norm{\mathbf{k}}\leq K} N\sum\limits_{|j|\leq
M}''\omega_j^{2s}\frac{\boldsymbol{\omega}^{2|\mathbf{k}|}}{\epsilon^{2[[\mathbf{k}]]}}
\sum\limits_{m=2}^N\Big(\frac{g^{(m)}(0)}{m!}\Big)^2
\sum\limits_{\mathbf{k}^1+\cdots+\mathbf{k}^m=\mathbf{k}}\Big(\frac{\epsilon^{[[\mathbf{k}^1]]+\cdots+[[\mathbf{k}^m]]}}{\boldsymbol{\omega}^{|\mathbf{k}^1|+\cdots+|\mathbf{k}^m|}} \Big)^2\\
&\sum\limits_{\mathbf{k}^1+\cdots+\mathbf{k}^m=\mathbf{k}}\Big(\sum\limits_{j_1+\cdots+j_m\equiv
j\ \textmd{mod}\ 2M}' \big(c\tilde{\xi}_{j_1}^{\mathbf{k}^1}\cdots
c\tilde{\xi}_{j_m}^{\mathbf{k}^m}\big)(\tau)\Big)^2.
\end{aligned}
\end{equation*}
According to the definition of $[[\cdot]]$, we have
\begin{equation*}
\begin{aligned}&[[\mathbf{k}^1]]+\cdots+[[\mathbf{k}^m]]=\frac{1}{2}(m+\norm{\mathbf{k}^1}+\cdots+\norm{\mathbf{k}^m})\geq\frac{1}{2}(m+\norm{\mathbf{k}})=
\frac{1}{2}(m-1)+[[\mathbf{k}]],\end{aligned}
\end{equation*}
which yields that
$\frac{\epsilon^{[[\mathbf{k}^1]]+\cdots+[[\mathbf{k}^m]]}}{\epsilon^{[[\mathbf{k}]]}}\leq
\epsilon^{\frac{1}{2}(m-1)}.$ Therefore, from Lemma \ref{lem 08}, it
follows that \begin{equation*}
\begin{aligned}\sum\limits_{\mathbf{k}\in \mathcal{K}}
\norm{f^{\mathbf{k}}(c\xi)}_s^2 &\leq N \epsilon\sum\limits_{m=
2}^N\Big(\frac{g^{(m)}(0)}{m!}\Big)^2 \epsilon^{(m-2)}C_{m,K,1}
\Big(\sum\limits_{\norm{\mathbf{k}}\leq
K}\norm{c\tilde{\xi}^{\mathbf{k}}}_s^2\Big)^m\\
&\leq C \epsilon P(||| c\tilde{\xi}|||_s^2),
\end{aligned}
\end{equation*}
where $P(x)=\sum\limits_{m=2}^N\Big(\frac{g^{(m)}(0)}{m!}\Big)^2
\epsilon^{(m-2)}C_{m,K,1}x^m.$

For the special case that $\mathbf{k} = \pm \langle j\rangle$, it
can be checked that
$[[\mathbf{k}^1]]+\cdots+[[\mathbf{k}^m]]\geq5/2$ for $m\geq 2$ by
considering $\mathbf{k}^1+\cdots+\mathbf{k}^m=\pm \langle j\rangle$.
Thus the bound restriction to this special case can be improved to a
factor $\epsilon^3$ instead of $\epsilon$ and then we have the
second result of \eqref{bounds f12}. \hfill\end{proof}

In a similar way, we consider    different rescaling
\begin{equation}\label{diff resca}
\begin{aligned}
\hat{c}\zeta_{j}^{\mathbf{k}}=\frac{\boldsymbol{\omega}^{s|\mathbf{k}|}}{\epsilon^{[[\mathbf{k}]]}}\zeta_j^{\mathbf{k}},\
\
\hat{c}\zeta^{\mathbf{k}}=\big(\hat{c}\zeta_{j}^{\mathbf{k}}\big)_{|j|\leq
M}=\frac{\boldsymbol{\omega}^{s|\mathbf{k}|}}{\epsilon^{[[\mathbf{k}]]}}\zeta^{\mathbf{k}}
\end{aligned}
\end{equation}
 in   $\mathrm{H}^1 = (H^1)^{\mathcal{K}}$ with norm
$|||\hat{c}\zeta|||_1^2=\sum\limits_{\norm{\mathbf{k}}\leq K}
\norm{\hat{c}\zeta^{\mathbf{k}}}_1^2$, where
$\hat{f}_j^{\mathbf{k}}$ is defined as $f_j^{\mathbf{k}}$ but with
$\boldsymbol{\omega}^{|\mathbf{k}|}$ replaced by
$\boldsymbol{\omega}^{s|\mathbf{k}|}$. We use similar notations
$\hat{c}\xi^{\mathbf{k}}\in \mathrm{H}^1$ and also get similar
bounds
\begin{equation*}
\begin{aligned} &\sum\limits_{\mathbf{k}\in \mathcal{K}}
\norm{\hat{f}^{\mathbf{k}}(\hat{c}\xi)}_1^2\leq C\epsilon
\hat{P}(||| \hat{c}\tilde{\xi}|||_1^2),\ \ \sum\limits_{|j|\leq M}
\norm{\hat{f}^{\pm\langle j\rangle }(\hat{c}\xi)}_1^2\leq
C\epsilon^3 \hat{P}_1(|||\hat{c}\tilde{\xi}|||_1^2)
\end{aligned}
\end{equation*}
with other functions $\hat{P}$ and $\hat{P}_1$.

\subsection{Reformulation of the reverse Picard iteration}\label{subsec:ref rev}
 In this subsection, we split  $c\zeta$ into two parts in the light of the two cases: $\mathbf{k}=\pm \langle j\rangle$
 and $\mathbf{k}
\neq\pm \langle j\rangle$ as follows:
\begin{equation}\label{abzeta}
\left\{\begin{aligned}
&a\zeta_j^{\mathbf{k}}=c\zeta_j^{\mathbf{k}}\qquad \textmd{if}\
\mathbf{k}=\pm \langle
j\rangle, \quad \textmd{and 0 else},\\
&b\zeta_j^{\mathbf{k}}=c\zeta_j^{\mathbf{k}}\qquad \textmd{if}\
\eqref{inequa}\ \textmd{is satisfied}, \quad \textmd{and 0 else}.
\end{aligned}\right.
\end{equation}
We remark that for  $a\zeta = (a\zeta_j^{\mathbf{k}} )\in
\mathrm{H}^s$ and $b\zeta = (b\zeta_j^{\mathbf{k}} )\in
\mathrm{H}^s$,  one has $a\zeta + b\zeta = c\zeta$ and
$|||a\zeta|||_s^2+|||b\zeta|||_s^2=|||c\zeta|||_s^2$. The same
denotation and  property are used for  $c\xi$.

We try to rewrite  the iterations \eqref{Pic ite j} and \eqref{Pic
ite notj} in an abstract form
\begin{equation}\label{Abstruct Pic ite}
\left\{\begin{aligned} &a\dot{\zeta}^{(n+1)}=\Omega^{-1}F(a\zeta^{(n)},b\zeta^{(n)})-Aa\zeta^{(n)}, \\
 &b\zeta^{(n+1)}=\Omega^{-1}\Psi G(a\zeta^{(n)},b\zeta^{(n)})-Bb\zeta^{(n)},
\end{aligned}\right.
\end{equation}
where \begin{equation*}
\begin{aligned} &(\Omega x)_j^{\mathbf{k}}=(\omega_j+|\mathbf{k}\cdot\boldsymbol{\omega}|) x_j^{\mathbf{k}},\quad  (\Psi x)_j^{\mathbf{k}}=2\phi_2(h^2\omega^2_j)\cos(\frac{1}{2}h(\mathbf{k}\cdot\boldsymbol{\omega}))
x_j^{\mathbf{k}},
\end{aligned}
\end{equation*}
and the operators $A, B$  are respectively defined as
\begin{equation*}
\begin{aligned} &(Aa\zeta)_j^{\pm \langle j\rangle}(\tau)=\frac{1}{\pm2\mathrm{i}s_{\langle j\rangle}h \epsilon}\Big(\frac{1}{2} \epsilon^2
h^2\sec(\frac{1}{2}h\omega_j) a\ddot{\zeta}_j^{\pm  \langle
j\rangle}+\cdots\Big), \\
&(Bb\zeta)_j^{\mathbf{k}}(\tau)= \frac{c_{\mathbf{k}}}{2s_{\langle
j\rangle+\mathbf{k}}s_{\langle j\rangle-\mathbf{k}}}\Big(\textmd{i}
\epsilon h
 \frac{s_{\mathbf{k}}(1+c_{\langle j\rangle+\mathbf{k}}c_{\langle
j\rangle-\mathbf{k}})}{c_{\mathbf{k}}^2}b\dot{\zeta}_j^{\mathbf{k}}+\cdots\Big)
\quad \textmd{for} \ (j, \mathbf{k}) \ \textmd{satisfying
}\eqref{inequa}.
\end{aligned}
\end{equation*}
The functions $F =(F ^{\mathbf{k}}_ j)$ and $G = (G ^{\mathbf{k}}_
j)$ are given respectively by
\begin{equation*}
\begin{aligned} &F_j^{\pm \langle
j\rangle}(a\zeta,b\zeta)=\frac{1}{\mp
\textmd{i}\epsilon}\frac{2\phi_2(h^2\omega^2_j)}{
\textmd{sinc}(\frac{1}{2}h\omega_j)}f_j^{\pm \langle
j\rangle}(c\xi),\ \
 G_j^{\mathbf{k}}(a\zeta,b\zeta)=-\frac{h^2(\omega_j+|\mathbf{k}\cdot\boldsymbol{\omega}|)}{4s_{\langle j\rangle+\mathbf{k}}s_{\langle
j\rangle-\mathbf{k}}}f_j^{\mathbf{k}}(c\xi)
\end{aligned}
\end{equation*}
for   $(j,k)$ satisfying \eqref{inequa}.

Since \begin{equation*}
\begin{aligned}\abs{\frac{1}{\pm2\mathrm{i}s_{\langle j\rangle}h \epsilon} \frac{1}{2} \epsilon^2
h^2\sec(\frac{1}{2}h\omega_j) }= \abs{\frac{\frac{1}{2}h
\epsilon}{\sin( h\omega_j)}}\leq\frac{1}{2}
\epsilon^{1/2},\end{aligned}
\end{equation*}
 the operator
$A$ is bounded by
\begin{equation*}
\begin{aligned} &|||(Aa\zeta)(\tau)|||_s\leq C
\sum\limits_{l=2}^Nh^{l-2}\epsilon^{l-3/2}|||\frac{d^l}{d\tau^l}(a\zeta)(\tau)|||_s.
\end{aligned}
\end{equation*}
We then compute
\begin{equation*}
\begin{aligned}&\abs{\frac{c_{\mathbf{k}}}{2s_{\langle j\rangle+\mathbf{k}}s_{\langle
j\rangle-\mathbf{k}}} \textmd{i} \epsilon h
 \frac{s_{\mathbf{k}}(1+c_{\langle j\rangle+\mathbf{k}}c_{\langle
j\rangle-\mathbf{k}})}{c_{\mathbf{k}}^2}}\leq \abs{\frac{\epsilon
h}{\epsilon^{1/2}h^2(
 \omega_j+|\mathbf{k} \cdot\boldsymbol{\omega}|)}
 \frac{s_{\mathbf{k}}(1+c_{\langle j\rangle+\mathbf{k}}c_{\langle
j\rangle-\mathbf{k}})}{c_{\mathbf{k}}}}\\
&\leq  \frac{\epsilon^{1/2}}{h}\frac{\frac{h}{2}|\mathbf{k} \cdot
\boldsymbol{\omega}|}{
 \omega_j+|\mathbf{k} \cdot\boldsymbol{\omega}|}
 \abs{\frac{1+c_{\langle j\rangle+\mathbf{k}}c_{\langle
j\rangle-\mathbf{k}}}{c_{\mathbf{k}}}}\leq C \epsilon^{1/2},
\end{aligned}
\end{equation*}
where we used $\abs{s_{\mathbf{k}}}\leq \frac{h}{2}|\mathbf{k}
\cdot\boldsymbol{\omega}|$. Thus  $B$ is bounded by
\begin{equation*}
\begin{aligned}
&|||(Bb\zeta)(\tau)|||_s\leq C
\epsilon^{1/2}|||(b\dot{\zeta})(\tau)|||_s+C
\sum\limits_{l=2}^Nh^{l-2}\epsilon^{l-1/2}|||\frac{d^l}{d\tau^l}(b\zeta)(\tau)|||_s.
\end{aligned}
\end{equation*}

From $\abs{\frac{2\phi_2(h^2\omega^2_j)}{
\textmd{sinc}(\frac{1}{2}h\omega_j)}}=\abs{
\textmd{sinc}(\frac{1}{2}h\omega_j)}\leq 1$ and \eqref{bounds f12},
it follows that   $|||F|||_s\leq C\epsilon^{1/2}$. By considering
\eqref{inequa} and \eqref{bounds f12}, one gets   $|||G|||_s\leq C$.
Moreover, according to \eqref{further-non-res cond}, we have
\begin{equation*}
\begin{aligned} |||\Psi^{-1}\Omega^{-1}F|||^2_s=&\sum\limits_{\mathbf{k}\in
\mathcal{K}} \sum\limits_{|j|\leq
M}''\omega_j^{2s}\abs{(\Psi^{-1}\Omega^{-1}F)_j^{\mathbf{k}}}^2
=\sum\limits_{\mathbf{k}\in \mathcal{K}} \sum\limits_{|j|\leq
M}''\omega_j^{2s}\abs{ \frac{h/2}{\epsilon\sin(h\epsilon)}}^2
\abs{f_j^{\pm \langle j\rangle}}^2\\
\leq&C\sum\limits_{\mathbf{k}\in \mathcal{K}} \sum\limits_{|j|\leq
M}''\omega_j^{2s}\abs{ \frac{1}{\epsilon ^{3/2}}}^2 \abs{f_j^{\pm
\langle j\rangle}}^2=C\frac{1}{\epsilon ^3} |||f^{\pm \langle
j\rangle}|||^2_s\leq C,
\end{aligned}
\end{equation*}
which implies $|||\Psi^{-1}\Omega^{-1}F|||_s\leq C$.

For the initial value condition \eqref{pica initial pl}, it  can be
rewritten as
\begin{equation}\label{abs ini con}
\begin{aligned} &a\zeta^{(n+1)}(0)=v+Pb\zeta^{(n)}(0)+Qb\zeta^{(n)}(0),
\end{aligned}
\end{equation}
where $ v_j^{\pm\langle j\rangle
}=\frac{\omega_j}{\epsilon}\Big(\frac{1}{2}q_j(0)\mp\frac{\textmd{i}}{2\omega_j}p_j(0)\Big)
$ and the operators $P$ and $Q$   are defined by
\begin{equation*}
\begin{aligned}
(Pb\zeta)_j^{\pm\langle
j\rangle}(0)=&-\frac{1}{2}\frac{\omega_j}{\epsilon}\sum\limits_{\mathbf{k}\neq
\pm\langle j\rangle}\frac{\epsilon^{[[\mathbf{k}]]}}{\boldsymbol{\omega}^{|\mathbf{k}|}}b\zeta_j^{\mathbf{k}}(0),\\
 (Qb\zeta)_j^{\pm\langle
j\rangle}(0)=&\mp\frac{1}{2
\omega_j}\frac{\omega_j}{\epsilon}\sum\limits_{\mathbf{k}\neq
\pm\langle
j\rangle}\frac{\epsilon^{[[\mathbf{k}]]}}{\boldsymbol{\omega}^{|\mathbf{k}|}}b\eta_j^{\mathbf{k}}(0).
\end{aligned}
\end{equation*}
From \eqref{initi cond}, it can be verified that   $v$ is bounded in
$\mathrm{H}^s$. For the bounds of the  operators $P$ and $Q$, we
have
\begin{equation*}
\begin{aligned} |||Pb\zeta(0)|||^2_s=&\sum\limits_{\mathbf{k}\in
\mathcal{K}} \sum\limits_{|j|\leq
M}''\omega_j^{2s}\abs{\frac{1}{2}\frac{\omega_j}{\epsilon}\sum\limits_{\mathbf{k}\neq
\pm\langle
j\rangle}\frac{\epsilon^{[[\mathbf{k}]]}}{\boldsymbol{\omega}^{|\mathbf{k}|}}b\zeta_j^{\mathbf{k}}(0)}^2\\
\leq&\frac{1}{4\epsilon^2}\sum\limits_{\mathbf{k}\in \mathcal{K}}
\sum\limits_{|j|\leq M}''\omega_j^{2s+2}
\Big(\sum\limits_{\mathbf{k}\neq \pm\langle
j\rangle}\frac{\epsilon^{2[[\mathbf{k}]]}}{\boldsymbol{\omega}^{2|\mathbf{k}|}}\Big)\Big(\sum\limits_{\mathbf{k}\neq
\pm\langle j\rangle}b\zeta_j^{\mathbf{k}}(0) ^2\Big)\\
\leq&\frac{1}{4}\sum\limits_{\mathbf{k}\in \mathcal{K}}
\sum\limits_{|j|\leq M}''\omega_j^{2s+2}
\Big(\sum\limits_{\mathbf{k}\neq \pm\langle j\rangle}
\boldsymbol{\omega}^{-2|\mathbf{k}|}\Big)
\Big(\sum\limits_{\mathbf{k}\neq \pm\langle
j\rangle}b\zeta_j^{\mathbf{k}}(0) ^2\Big)\\
\leq& C |||\Omega b\zeta(0)|||^2_{s}\leq  C |||
b\zeta(0)|||^2_{s+1},
\end{aligned}
\end{equation*}
and similarly
\begin{equation*}
\begin{aligned} |||(Qb\zeta)(0)|||^2_s \leq C |||
b\eta(0)|||^2_{s}.
\end{aligned}
\end{equation*}
Thence the bounds $|||(Pb\zeta)(0)|||_s\leq C$ and $
|||(Qb\zeta)(0)|||_s\leq C$  are obtained.   The starting iterates
of \eqref{abs ini con} are chosen as $a\zeta^{(0)}(\tau ) = v$ and
$b\zeta^{(0)}(\tau ) = 0$.

\subsection{Bounds of the coefficient functions}\label{subsec:bou coe}
\begin{prop}\label{pro: functions} The modulation
functions $\zeta^{\mathbf{k}}$  of \eqref{MFE-AAVF}  are bounded by
\begin{equation}\begin{aligned}
&\sum\limits_{\norm{\mathbf{k}}\leq2N}\Big(\frac{\boldsymbol{\omega}^{|\mathbf{k}|}}{\epsilon^{[[\mathbf{k}]]}}\norm{\zeta^{\mathbf{k}}(\epsilon
t)}_s\Big)^2\leq C
\end{aligned}
\label{bound modu func}%
\end{equation}
and the same bound holds for any fixed number of derivatives of
 $\zeta^{\mathbf{k}}$ with respect to the slow time  $\tau = \epsilon
t$.
\end{prop}

\begin{proof}
On the basis of the above analysis and by induction, it is easy to
prove that the iterates $a\zeta^{(n)},\ b\zeta^{(n)}$ and their
derivatives with respect to  $\tau$ are bounded in $\mathrm{H}^s$
for $0 \leq \tau\leq 1$ and $n \leq 4N$.   These bounds
  show that $c\zeta^{(n)}=a\zeta^{(n)}+ b\zeta^{(n)}$ is bounded in
$\mathrm{H}^s$ and then the bound  \eqref{bound modu func} is
obtained. \hfill\end{proof}

\begin{prop}\label{pro: bound of the e}
  The expansion \eqref{MFE-AAVF} is bounded by
\begin{equation}
\norm{\tilde{\mathbf{q}}(t)}_{s+1}+\norm{\tilde{\mathbf{p}}(t)}_{s}\leq
C\epsilon\quad for \quad 0\leq t\leq \epsilon^{-1}.
\label{bound TMFE}%
\end{equation}
 For $|j|\leq M$, it further holds that
\begin{equation}\begin{aligned}
&\tilde{q}_j(t)=\zeta_j^{\langle j\rangle}(\epsilon t)
\mathrm{e}^{\mathrm{i} \omega_j t}+\zeta_j^{-\langle
j\rangle}(\epsilon t) \mathrm{e}^{-\mathrm{i} \omega_j t}+r_j \quad
\textmd{with} \quad \norm{\mathbf{r}}_{s+1}\leq C\epsilon^2.
\end{aligned}
\label{jth TMFE}
\end{equation}
If the condition  \eqref{another-non-res cond} fails to be
satisfied, then the bound is $\norm{r}_{s+1}\leq C\epsilon^{3/2}$.
\end{prop}
\begin{proof}
More   precisely,   the following bounds for the $(4N)$-th iterates
can be obtained
\begin{equation}\label{Bounds func zeta}
\begin{array}{ll}   |||a\zeta(0)|||_s\leq C,\ \ &|||\Omega a\dot{\zeta}(\tau)|||_s\leq
C\epsilon^{1/2},\\
|||\Psi^{-1}a\dot{\zeta}(\tau)|||_s\leq C,\ \ &|||\Psi^{-1}\Omega b
\zeta (\tau)|||_s\leq C,
\end{array}
\end{equation}
where $C$ depends on $N$, but not on $\epsilon, h, M$. From these
results \eqref{Bounds func zeta}, it follows that
\begin{equation*}
\begin{aligned} |||a\dot{\zeta}  |||_{s+1}&=|||\Omega a\dot{\zeta}  |||_{s}\leq
C\epsilon^{1/2},\\
|||b\zeta  |||^2_{s+1}&=\sum\limits_{\mathbf{k}\in \mathcal{K}}
\sum\limits_{|j|\leq M}''\omega_j^{2s+2}\abs{b\zeta_j
}^2=\sum\limits_{\mathbf{k}\in \mathcal{K}} \sum\limits_{|j|\leq
M}''\omega_j^{2s}\frac{\omega_j^{2}}{(\omega_j+|\mathbf{k}\cdot
\boldsymbol{\omega}|)^2}\abs{(\omega_j+|\mathbf{k}\cdot\boldsymbol{\omega}|)b\zeta_j }^2\\
&\leq |||\Omega b\zeta  (\tau)|||^2_{s}\leq C.
\end{aligned}
\end{equation*}
Hence, we have
\begin{equation*}
\begin{aligned}|||c\zeta (\tau)-a\zeta(0)|||_{s+1}=|||a\zeta (\tau)+b\zeta
(\tau)-a\zeta(0)|||_{s+1}  \leq |||a\dot{\zeta} |||_{s+1}+|||b\zeta
|||_{s+1}\leq C.
\end{aligned}
\end{equation*}
Then according to the fact that
$\zeta_j^{\mathbf{k}}=\frac{\epsilon^{[[\mathbf{k}]]}}{\boldsymbol{\omega}^{|\mathbf{k}|}}(c\zeta_j^{\mathbf{k}}-a\zeta_j^{\mathbf{k}}(0)+a\zeta_j^{\mathbf{k}}(0)),
$  it is yielded that
\begin{equation*}
\begin{aligned}
&|||\tilde{\mathbf{q}}  |||^2_{s+1}=\sum\limits_{\mathbf{k}\in
\mathcal{K}} \sum\limits_{|j|\leq
M}''\omega_j^{2s+2}\abs{\sum\limits_{\norm{\mathbf{k}}\leq 2N}
\mathrm{e}^{\mathrm{i}(\mathbf{k} \cdot\boldsymbol{\omega}) t}\zeta_j^{\mathbf{k}}}^2\\
\leq&\sum\limits_{\mathbf{k}\in \mathcal{K}} \sum\limits_{|j|\leq
M}''\omega_j^{2s+2}\Big[\frac{\epsilon}{\omega_j}\big(\abs{a\zeta_j^{\langle
j\rangle}(0)}+\abs{a\zeta_j^{-\langle j\rangle}(0)}\big)
+\sum\limits_{\norm{\mathbf{k}}\leq 2N}\frac{\epsilon^{[[\mathbf{k}]]}}{\boldsymbol{\omega}^{|\mathbf{k}|}}\abs{c\zeta_j^{\mathbf{k}}-a\zeta_j^{\mathbf{k}}(0)}\Big]^2\\
\leq&2\epsilon^2\sum\limits_{\mathbf{k}\in \mathcal{K}}
\sum\limits_{|j|\leq M}''\omega_j^{2s}\Big(  \abs{a\zeta_j^{\langle
j\rangle}(0)}+\abs{a\zeta_j^{-\langle j\rangle}(0)}\Big)^2\\
& +2\sum\limits_{\mathbf{k}\in \mathcal{K}}\sum\limits_{|j|\leq
M}''\omega_j^{2s+2}\Big(
\sum\limits_{\norm{\mathbf{k}}\leq 2N}\frac{\epsilon^{[[\mathbf{k}]]}}{\boldsymbol{\omega}^{|\mathbf{k}|}}\abs{c\zeta_j^{\mathbf{k}}-a\zeta_j^{\mathbf{k}}(0)}\Big)^2\\
\leq&4\epsilon^2 |||a\zeta (0)|||_s^2 +2\sum\limits_{\mathbf{k}\in
\mathcal{K}}\sum\limits_{|j|\leq M}''\omega_j^{2s+2}
\Big(\sum\limits_{\norm{\mathbf{k}}\leq 2N}\frac{\epsilon^{2[[\mathbf{k}]]}}{\boldsymbol{\omega}^{2|\mathbf{k}|}}\Big)\Big(\sum\limits_{\norm{\mathbf{k}}\leq 2N}\abs{c\zeta_j^{\mathbf{k}}-a\zeta_j^{\mathbf{k}}(0)}^2 \Big)\\
\leq&4\epsilon^2 |||a\zeta (0)|||_s^2
+2C_{K,1}\epsilon^2|||c\zeta-a\zeta(0)|||^2_{s+1} \leq C \epsilon^2.
\end{aligned}
\end{equation*}
By \eqref{qp connec} and similar analysis, it can be proved that
$|||\tilde{\mathbf{p}} |||_{s}\leq C \epsilon$. Thus   the bound
\eqref{bound TMFE} is true.

From \eqref{bounds f12} and \eqref{Abstruct Pic ite}, it follows
that $\Big(\sum\limits_{\norm{\mathbf{k}}=1}\norm{(\Psi^{-1}\Omega b
\zeta)^{\mathbf{k}}}_s^2\Big)^{1/2}\leq C \epsilon$ for $b \zeta=(b
\zeta)^{(4N)}$. Moreover, in terms  of \eqref{another-non-res cond},
it is obtained that $$ \sum\limits_{|j|\leq
M}\sum\limits_{j_1+j_2=j}\sum\limits_{\mathbf{k}=\pm\langle
j_1\rangle\pm\langle j_2\rangle}
 \omega_j^{2(s+1)}|b\zeta_j^{\mathbf{k}}|^2 \leq C \epsilon.$$
  These bounds  as well as   \eqref{Bounds func zeta} lead to
  \eqref{jth TMFE}.
\hfill\end{proof}

 For the alternative scaling \eqref{diff resca}, one can obtain the
same bounds
\begin{equation}\label{Bounds  zeta sca}
\begin{aligned} & |||\hat{a}\zeta(0)|||_1\leq C,\ \ |||\Omega \hat{a}\dot{\zeta}(\tau)|||_1\leq
C\epsilon^{1/2},\ \  |||\Psi^{-1}\Omega \hat{b} \zeta
(\tau)|||_1\leq C.
\end{aligned}
\end{equation}
The following bound is also true   for this scaling:
\begin{equation}\label{Bounds  eta sca}
\begin{aligned}
&\Big(\sum\limits_{\norm{\mathbf{k}}=1}\norm{(\Psi^{-1}\Omega
\hat{b} \zeta)^{\mathbf{k}}}_1^2\Big)^{1/2}\leq C \epsilon.
\end{aligned}
\end{equation}

\subsection{Defects}\label{subsec:def}
We express the defect in \eqref{AAVFmethod} as another form
\begin{equation}\label{defects}
\begin{aligned}
&\delta_j(t)=\frac{\tilde{q}_j(t+h)-2\cos(h\omega_j)\tilde{q}_j(t)+\tilde{q}_j(t-h)}{h^2\phi_2(h^2\omega^2_j)}\\
&-\Big[\displaystyle\int_{0}^{1}f_j((1-\sigma)\tilde{\mathbf{q}}_{h}(t)+\sigma
\tilde{\mathbf{q}}_{h}(t+h))d\sigma+
\displaystyle\int_{0}^{1}f_j((1-\sigma)\tilde{\mathbf{q}}_{h}(t-h)+\sigma
\tilde{\mathbf{q}}_{h}(t))d\sigma\Big],
\end{aligned}
\end{equation}
where $\tilde{q}_j$  is given in \eqref{MFE-AAVF} with
$\zeta^{\mathbf{k}}_j=(\zeta^{\mathbf{k}}_j)^{(4N)}$  obtained after
$4N$ iterations of the procedure in Sect. \ref{subsec:rev pic}. This
defect can also be rewritten  as
\begin{equation*}
\begin{aligned}
&\delta_j(t)=\sum\limits_{\norm{\mathbf{k}}\leq
NK}\mathbf{d}^{\mathbf{k}}(\epsilon t)
e^{\textmd{i}(\mathbf{k}\cdot\boldsymbol{\omega}) t}+R(t),
\end{aligned}
\end{equation*}
where
\begin{equation}\label{djk}
\begin{aligned}
d_j^{\mathbf{k}}=&\frac{1}{h^2\phi_2(h^2\omega^2_j)}\tilde{L}_j^{\mathbf{k}}
\zeta_j^{\mathbf{k}}+\sum\limits_{m= 2}^N\frac{g^{(m)}(0)}{m!}
\sum\limits_{\mathbf{k}^1+\cdots+\mathbf{k}^m=\mathbf{k}}\sum\limits_{j_1+\cdots+j_m\equiv
j\
\textmd{mod}\ 2M}'\\
&\int_{0}^{1} \Big[\big(\xi_{j_1}^{\mathbf{k}^1}\cdot\ldots\cdot
\xi_{j_m}^{\mathbf{k}^m}\big)(t\epsilon,\sigma)\Big]d\sigma.
\end{aligned}
\end{equation}
It is noted that here we consider  $\norm{\mathbf{k}}\leq NK$ for
$d_j^{\mathbf{k}}$, and it is assumed that
$\zeta_j^{\mathbf{k}}=\eta_j^{\mathbf{k}}=0$ for
$\norm{\mathbf{k}}>K:=2N$. The truncation of the operator
$L_j^{\mathbf{k}}$ after the $\epsilon^N$ term is denoted by
$\tilde{L}_j^{\mathbf{k}}$. The remainder terms of the Taylor
expansion of $f$ after $N$ terms are contained in $R$. By  the bound
\eqref{bound TMFE} and the estimates \eqref{Bounds func zeta}, it is
true that
 $\norm{R}_{s+1} \leq
C\epsilon^{N+1}$.

  Using  Cauchy-Schwarz inequality and  Lemma \ref{lem 08},
we obtain
 \begin{equation*}
\begin{aligned}
&\norm{\sum\limits_{\norm{\mathbf{k}}\leq
NK}\mathbf{d}^{\mathbf{k}}(\epsilon t)
e^{\textmd{i}(\mathbf{k}\cdot\boldsymbol{\omega}) t}}_{s}^2=
\sum\limits_{|j|\leq
M}''\omega_j^{2s}\abs{\sum\limits_{\norm{\mathbf{k}}\leq
NK}d_j^{\mathbf{k}}
e^{\textmd{i}(\mathbf{k}\cdot\boldsymbol{\omega}) t}}^2\\
 = & \sum\limits_{|j|\leq
M}''\omega_j^{2s}\abs{\sum\limits_{\norm{\mathbf{k}}\leq
NK}\boldsymbol{\omega}^{-|\mathbf{k}|}(\boldsymbol{\omega}^{|\mathbf{k}|}d_j^{\mathbf{k}} e^{\textmd{i}(\mathbf{k}\cdot\boldsymbol{\omega}) t})}^2\\
\leq & \sum\limits_{|j|\leq M}''\omega_j^{2s}\Big(
\sum\limits_{\norm{\mathbf{k}}\leq
NK}\boldsymbol{\omega}^{-2|\mathbf{k}|}\Big)\Big(\sum\limits_{\norm{\mathbf{k}}\leq
NK}(\boldsymbol{\omega}^{|\mathbf{k}|}d_j^{\mathbf{k}} )^2\Big)\\
\leq& C_{NK,1} \sum\limits_{\norm{\mathbf{k}}\leq
NK}\norm{\boldsymbol{\omega}^{|\mathbf{k}|}\mathbf{d}^{\mathbf{k}}(\epsilon
t)}_{s}^2.
\end{aligned}
\end{equation*}
 The right-hand side of this result can be estimated as
follows.
\begin{prop}\label{pro: bound of d}
It is obtained that $ \sum\limits_{\norm{\mathbf{k}}\leq
NK}\norm{\boldsymbol{\omega}^{|\mathbf{k}|}\mathbf{d}^{\mathbf{k}}(\epsilon
t)}_{s}^2\leq C\epsilon^{2(N+1)}. $

\end{prop}
\begin{proof}
We will prove this result for three cases: truncated, near-resonant
and non-resonant modes.

 $\bullet$ \textbf{Truncated    and near-resonant modes.} For truncated modes
($\zeta_j^{\mathbf{k}}=\eta_j^{\mathbf{k}}=0$ for
$\norm{\mathbf{k}}>K:=2N$) and near-resonance modes ($(j,
\mathbf{k}) \in\mathcal{R}_{\epsilon,h}$), the defect is of  the
same form
\begin{equation*}
\begin{aligned}
d_j^{\mathbf{k}}=&\sum\limits_{m= 2}^N\frac{g^{(m)}(0)}{m!}
\sum\limits_{\mathbf{k}^1+\cdots+\mathbf{k}^m=\mathbf{k}}\sum\limits_{j_1+\cdots+j_m\equiv
j\ \textmd{mod}\ 2M}'\int_{0}^{1}
\Big[\big(\xi_{j_1}^{\mathbf{k}^1}\cdot\ldots\cdot
\xi_{j_m}^{\mathbf{k}^m}\big)(t\epsilon,\sigma)\Big]d\sigma.
\end{aligned}
\end{equation*}

For truncated modes  the defect is rewritten  as
$d_j^{\mathbf{k}}=\epsilon^{[[\mathbf{k}]]}\boldsymbol{\omega}^{-|\mathbf{k}|}
f_j^{\mathbf{k}} (c\xi). $ By \eqref{bounds f12}, and \eqref{Bounds
func zeta} with $NK$ instead of $K$, the bound $|||f|||_s^2\leq C
\epsilon$ is clear, which leads to
\begin{equation*}
\begin{aligned}
&\sum\limits_{\norm{\mathbf{k}}>K}\sum\limits_{|j|\leq
M}'\omega_j^{2s}|\boldsymbol{\omega}^{|\mathbf{k}|}d_j^{\mathbf{k}}|^2\leq
\sum\limits_{\norm{\mathbf{k}}>K}\sum\limits_{|j|\leq
M}'\omega_j^{2s}|f_j^{\mathbf{k}}|^2\epsilon^{2[[\mathbf{k}]]}\leq
C\epsilon^{2(N+1)},
\end{aligned}
\end{equation*}
where the fact that $2[[\mathbf{k}]]= \norm{\mathbf{k}} + 1 \geq K +
2 = 2(N + 1)$ is used.

 For the near-resonant modes   the  defect can be expressed in    the rescaling
\eqref{diff resca} as $d_j^{\mathbf{k}}=
\epsilon^{[[\mathbf{k}]]}\boldsymbol{\omega}^{-s|\mathbf{k}|}
\hat{f}_j^{\mathbf{k}}(\hat{c}\xi). $ Therefore, by considering
$|||\hat{f}|||_1^2\leq C \epsilon$ and the non-resonance condition
\eqref{non-resonance cond} with $\sigma =s - 1$, it is easy to check
that
\begin{equation*}
\begin{aligned}
 \sum\limits_{(j,k)\in
\mathcal{R}_{\epsilon,h}}\omega_j^{2s}|\boldsymbol{\omega}^{|\mathbf{k}|}d_j^{\mathbf{k}}|^2&=\sum\limits_{(j,k)\in
\mathcal{R}_{\epsilon,h}}\frac{\omega_j^{2(s-1)}}{\boldsymbol{\omega}^{2(s-1)|\mathbf{k}|}}\epsilon^{2[[\mathbf{k}]]}\omega_j^2|\hat{f}_j^{\mathbf{k}}|^2\\
&\leq C\sup_{(j,k)\in
\mathcal{R}_{\epsilon,h}}\frac{\omega_j^{2(s-1)}}{\boldsymbol{\omega}^{2(s-1)|\mathbf{k}|}}\epsilon^{2[[\mathbf{k}]]+1}\leq
C\epsilon^{2(N+1)}.
\end{aligned}
\end{equation*}

 $\bullet$ \textbf{Non-resonant mode.} For  the
non-resonant mode ($\norm{\mathbf{k}}>K$ and   $(j, \mathbf{k})$
satisfies \eqref{inequa}),   we reformulate the defect   in the
scaled variables of Sect. \ref{subsec:res est}  as
\begin{equation*}
\begin{aligned}
\boldsymbol{\omega}^{|\mathbf{k}|}d_j^{\mathbf{k}}=&
\epsilon^{[[\mathbf{k}]]}
\Big(\frac{1}{h^2\phi_2(h^2\omega^2_j)}\tilde{L}_j^{\mathbf{k}}
c\zeta_j^{\mathbf{k}}+f_j^{\mathbf{k}}(c\xi)\Big).
\end{aligned}
\end{equation*}
Splitting  them into  $\mathbf{k}={\pm\langle j\rangle}$ and
$\mathbf{k}\neq{\pm\langle j\rangle}$ yields
\begin{equation*}
\begin{array}{ll}
\omega_jd_j^{\pm\langle j\rangle}= \epsilon \Big(\pm
\textmd{i}\epsilon \omega_j
\frac{\textmd{sinc}(h\omega_j/2)}{\phi_2(h^2\omega^2_j)}
\big(a\dot{\zeta}_j^{\pm\langle j\rangle}+(Aa\zeta)_j^{\pm\langle
j\rangle}\big)+f_j^{\pm\langle j\rangle} (c\xi)\Big),\\
\boldsymbol{\omega}^{|\mathbf{k}|}d_j^{\mathbf{k}}=\epsilon^{[[\mathbf{k}]]}
\Big(\frac{2s_{\langle j\rangle+\mathbf{k}}s_{\langle
j\rangle-\mathbf{k}}}{h^2c_{\mathbf{k}}\phi_2(h^2\omega^2_j)}\big(b\zeta_j^{\mathbf{k}}+(Bb\zeta)_j^{\mathbf{k}}\big)+f_j^{\mathbf{k}}(c\xi)\Big).
\end{array}
\end{equation*}
It is noted that the functions here are actually the $4N$-th
iterates of the iteration in Sect.   \ref{subsec:rev pic}. By
expressing $f_j^{\pm\langle j\rangle}$ and $f_j^{\mathbf{k}}$ in
terms of $F,G$ and inserting them from \eqref{Abstruct Pic ite} into
this defect, it is arrived at that
\begin{equation*}
\begin{array}{ll}
\omega_jd_j^{\pm\langle j\rangle}= 2 \omega_j \alpha_j^{\pm\langle
j\rangle}\big( \big[a\dot{\zeta}_j^{\pm\langle
j\rangle}\big]^{(4N)}-\big[a\dot{\zeta}_j^{\pm\langle
j\rangle}\big]^{(4N+1)}\big), \ \ \ & \alpha_j^{\pm\langle
j\rangle}=\pm  \textmd{i}\epsilon^2
\frac{\textmd{sinc}(h\omega_j/2)}{2\phi_2(h^2\omega^2_j)},\\
\boldsymbol{\omega}^{|\mathbf{k}|}d_j^{\mathbf{k}}=
\beta_j^{\mathbf{k}}\big([b\zeta_j^{\mathbf{k}}]^{(4N)}-
[b\zeta_j^{\mathbf{k}} ]^{(4N+1)}\big), \ \ \ &
\beta_j^{\mathbf{k}}=\epsilon^{[[\mathbf{k}]]}\frac{2s_{\langle
j\rangle+\mathbf{k}}s_{\langle
j\rangle-\mathbf{k}}}{h^2c_{\mathbf{k}}\phi_2(h^2\omega^2_j)}.
\end{array}
\end{equation*}
Looking closer at these expressions, we  consider new variables
given as
\begin{equation*}
\begin{array}{ll}
 \tilde{a}\zeta_j^{\pm\langle j\rangle}=\alpha_j^{\pm\langle
j\rangle} a \zeta_j^{\pm\langle j\rangle},\ \
\tilde{b}\zeta_j^{\mathbf{k}}=\beta_j^{\mathbf{k}} b
\zeta_j^{\mathbf{k}}
\end{array}
\end{equation*}
and rewrite the iteration \eqref{Abstruct Pic ite} in these
variables as
\begin{equation*}
\begin{aligned} &\tilde{a}\dot{\zeta}^{(n+1)}=\Omega^{-1}\tilde{F}(\tilde{a}\zeta^{(n)},\tilde{b}\zeta^{(n)})-A\tilde{a}\zeta^{(n)}, \\
 &\tilde{b}\zeta^{(n+1)}=
 \tilde{G}(\tilde{a}\zeta^{(n)},\tilde{b}\zeta^{(n)})-B\tilde{b}\zeta^{(n)}.
\end{aligned}
\end{equation*}
Here the transformed functions are defined by
\begin{equation*}
\begin{array}{ll}
 \tilde{F}_j^{\pm\langle j\rangle}(\tilde{a}\zeta,\tilde{b}\zeta)=\alpha_j^{\pm\langle
j\rangle} F_j^{\pm\langle
j\rangle}(\alpha^{-1}\tilde{a}\zeta,\beta^{-1}\tilde{b}\zeta)=-\epsilon
f_j^{\pm\langle
j\rangle}(\alpha^{-1}\tilde{a}\zeta+\beta^{-1}\tilde{b}\zeta),\\
\tilde{G}_j^{\mathbf{k}}(\tilde{a}\zeta,\tilde{b}\zeta)=\beta_j^{\mathbf{k}}(\Psi
\Omega^{-1}
G)_j^{\mathbf{k}}(\alpha^{-1}\tilde{a}\zeta,\beta^{-1}\tilde{b}\zeta)=-\epsilon^{[[\mathbf{k}]]}
f_j^{\mathbf{k}}(\alpha^{-1}\tilde{a}\zeta+\beta^{-1}\tilde{b}\zeta).
\end{array}
\end{equation*}
In the iteration for the initial values, one has
\begin{equation*}
\begin{aligned}
  &\tilde{a} \zeta^{(n+1)}(0)=\alpha v
  +\tilde{P}\tilde{b}\zeta^{(n)}(0)+\tilde{Q}\tilde{b}\zeta^{(n)}(0),
\end{aligned}
\end{equation*}
where $\tilde{P} = \alpha P\beta^{-1}, \tilde{Q} = \alpha
Q\beta^{-1}$.  For the bound of  $\tilde{P}$, we have
\begin{equation*}
\begin{aligned}
&|||\tilde{P}\tilde{b}\zeta(0) |||^2_{s}\\
=&\sum\limits_{\mathbf{k}\in \mathcal{K}} \sum\limits_{|j|\leq
M}''\omega_j^{2s}\abs{\textmd{i}\epsilon^2
\frac{\textmd{sinc}(h\omega_j/2)}{2\phi_2(h^2\omega^2_j)}
\frac{1}{2}\frac{\omega_j}{\epsilon}\sum\limits_{\mathbf{k}\neq
\pm\langle
j\rangle}\frac{h^2c_{\mathbf{k}}\phi_2(h^2\omega^2_j)}{\epsilon^{[[\mathbf{k}]]}2s_{\langle
j\rangle+\mathbf{k}}s_{\langle
j\rangle-\mathbf{k}}}\frac{\epsilon^{[[\mathbf{k}]]}}{\boldsymbol{\omega}^{|\mathbf{k}|}}\tilde{b}\zeta_j^{\mathbf{k}}(0)
}^2\\
\leq& \frac{\epsilon^2h^4}{64}\sum\limits_{\mathbf{k}\in
\mathcal{K}} \sum\limits_{|j|\leq M}''\omega_j^{2s}
\Big(\sum\limits_{\mathbf{k}\neq \pm\langle
j\rangle}\frac{\omega_j}{\abs{s_{\langle
j\rangle+\mathbf{k}}s_{\langle
j\rangle-\mathbf{k}}}}\boldsymbol{\omega}^{-|\mathbf{k}|}\tilde{b}\zeta_j^{\mathbf{k}}(0)\Big)^2
\\
\leq& \frac{\epsilon^2h^4}{64}\sum\limits_{\mathbf{k}\in
\mathcal{K}} \sum\limits_{|j|\leq M}''\omega_j^{2s}
\Big(\sum\limits_{\mathbf{k}\neq \pm\langle
j\rangle}\frac{1}{\epsilon^{1/2}h^2}\boldsymbol{\omega}^{-|\mathbf{k}|}\tilde{b}\zeta_j^{\mathbf{k}}(0)\Big)^2\\
\leq& \frac{\epsilon }{64}\sum\limits_{\mathbf{k}\in \mathcal{K}}
\sum\limits_{|j|\leq M}''\omega_j^{2s}
\Big(\sum\limits_{\mathbf{k}\neq \pm\langle j\rangle}
\boldsymbol{\omega}^{-2|\mathbf{k}|}\sum\limits_{\mathbf{k}\neq
\pm\langle j\rangle} (\tilde{b}\zeta_j^{\mathbf{k}}(0))^2\Big)\leq C
\epsilon ||| \tilde{b}\zeta(0) |||^2_{s}.
\end{aligned}
\end{equation*}
Similarly, it is obtained that
\begin{equation*}
\begin{aligned}
&|||\tilde{Q}\tilde{b}\zeta(0) |||^2_{s} \leq C \epsilon |||
\tilde{b}\zeta(0) |||^2_{s}.
\end{aligned}
\end{equation*}

It can be verified that in an $H^s$-neighbourhood of $0$ where the
bounds \eqref{Bounds func zeta} hold, the partial derivatives of
$\tilde{F}$ with respect to $\tilde{a}\zeta$ and $\tilde{b}\zeta$
 are
bounded by $\mathcal{O}(\epsilon^{1/2})$. Meanwhile,  the partial
derivative   of $\tilde{G}$ with respect to $\tilde{b}\zeta$  is
bounded by $\mathcal{O}(\epsilon^{1/2})$ but that of $\tilde{G}$
with respect to $\tilde{a}\zeta$ is only $\mathcal{O}(1)$. It is
noted that these results are the same as those described in Sect.
6.9 of \cite{Cohen08-1}. Likewise, we obtain
\begin{equation*}
\begin{aligned}
&||| \Omega (\tilde{a}\dot{\zeta} ^{(4N+1)}-\tilde{a}\dot{\zeta}
^{(4N)})|||_{s}\leq C\epsilon^{N+2},\\
&|||  \tilde{b} \zeta ^{(4N+1)}-\tilde{b} \zeta
^{(4N)})|||_{s}\leq C\epsilon^{N+2},\\
&|||  \tilde{a} \zeta (0)^{(4N+1)}-\tilde{a} \zeta(0)
^{(4N)})|||_{s}\leq C\epsilon^{N+2}.
\end{aligned}
\end{equation*}
Thus,   it is yielded for $\tau\leq1 $ and $(j,
\mathbf{k})\in\mathcal{R}_{\epsilon,h}$ that
\begin{equation}\label{right hand esti non}
\begin{aligned}
&\Big( \sum\limits_{\norm{\mathbf{k}}\leq
K}\norm{\boldsymbol{\omega}^{|\mathbf{k}|}\mathbf{d}^{\mathbf{k}}(\tau)}_{s}^2\Big)^{1/2}\leq
C\epsilon^{N+1}.
\end{aligned}
\end{equation}

With \eqref{right hand esti non},  the defect   \eqref{defects} has
the bound $\norm{\delta(t)}_{s}\leq C\epsilon^{N+1}$   for $t\leq
\epsilon^{-1}.$ For  the defect in the initial conditions
\eqref{initial pl} and \eqref{initial mi}, it holds that $$
\norm{\mathbf{q}(0)-\tilde{\mathbf{q}}(0)}_{s+1}+\norm{\mathbf{p}(0)-\tilde{\mathbf{p}}(0)}_{s}\leq
C\epsilon^{N+1}.$$ With the alternative scaling \eqref{diff resca},
the following result is derived
\begin{equation}\label{alter esti non}
\begin{aligned}
&\Big( \sum\limits_{\norm{\mathbf{k}}\leq
K}\norm{\boldsymbol{\omega}^{s|\mathbf{k}|}\mathbf{d}^{\mathbf{k}}(\tau)}_{1}^2\Big)^{1/2}\leq
C\epsilon^{N+1}.
\end{aligned}
\end{equation}
\hfill\end{proof}

\subsection{Remainders}\label{subsec:rem}
\begin{prop}\label{pro: remain}For the difference of
the numerical solution and its modulated Fourier expansion, we have
\begin{equation}
\norm{\mathbf{q}_{n}-\tilde{\mathbf{q}}(t)}_{s+1}+\norm{\mathbf{p}_{n}-\tilde{\mathbf{p}}(t)}_{s}\leq
C\epsilon^N\quad for \quad 0\leq t=nh\leq \epsilon^{-1}.
\label{error MFE}%
\end{equation}
\end{prop}

\begin{proof}
With the notations $\Delta
\mathbf{q}_{n}=\tilde{\mathbf{q}}(t_n)-\mathbf{q}_{n},\ \Delta
\mathbf{p}_{n}=\tilde{\mathbf{p}}(t_n)-\mathbf{p}_{n}$, one gets
\begin{equation*}
\begin{aligned}
\left(
  \begin{array}{c}
    \Delta \mathbf{q}_{n+1} \\
    \Omega^{-1}\Delta \mathbf{p}_{n+1} \\
  \end{array}
\right)=\left(
          \begin{array}{cc}
            \cos(h\Omega) & \sin(h\Omega) \\
            -\sin(h\Omega) & \cos(h\Omega) \\
          \end{array}
        \right)\left(
  \begin{array}{c}
    \Delta \mathbf{q}_{n} \\
    \Omega^{-1}\Delta \mathbf{p}_{n} \\
  \end{array}
\right)+h\left(
          \begin{array}{c}
           h\Omega\phi_2(V)\Omega^{-1}( \Delta f +\delta)\\
           \phi_1(V) \Omega^{-1} (\Delta f + \delta) \\
          \end{array}
        \right),
\end{aligned}
\end{equation*}
where $$\Delta f=\int_{0}^{1}\big(f((1-\sigma)\mathbf{q}_{n}+\sigma
\mathbf{q}_{n+1})-f((1-\sigma)\tilde{\mathbf{q}}(t_n)+\sigma
\tilde{\mathbf{q}}(t_n+h))\big)d\sigma .$$ Using the Lipschitz bound
given in Sect. 4.2 of \cite{Hairer08} and Sect. 6.10 of
\cite{Cohen08-1}, we have
\begin{equation*}
\begin{aligned}
\norm{\Omega^{-1}\Delta f}_{s+1}=\norm{\Delta f}_{s}\leq \epsilon
(\norm{\Delta \mathbf{q}_{n}}_{s}+\norm{\Delta
\mathbf{p}_{n}}_{s-1}).
\end{aligned}
\end{equation*}
Moreover, it is clear that
$\norm{\Omega^{-1}\delta(t)}_{s+1}=\norm{\delta(t)}_{s}\leq
C\epsilon^{N+1}$. Therefore, we obtain
\begin{equation*}
\begin{aligned}
\norm{\left(
  \begin{array}{c}
    \Delta \mathbf{q}_{n+1} \\
    \Omega^{-1}\Delta \mathbf{p}_{n+1} \\
  \end{array}
\right) }_{s+1}\leq\norm{ \left(
  \begin{array}{c}
    \Delta \mathbf{q}_{n} \\
    \Omega^{-1}\Delta \mathbf{p}_{n} \\
  \end{array}
\right)}_{s+1}+h\Big(C\epsilon \norm{ \Delta
\mathbf{q}_{n}}_{s+1}+C\epsilon \norm{\Delta
\mathbf{p}_{n}}_{s+1}+C\epsilon^{N+1}\Big),
\end{aligned}
\end{equation*}
which implies $ \norm{ \Delta
\mathbf{q}_{n}}_{s+1}+\norm{\Omega^{-1} \Delta
\mathbf{p}_{n}}_{s+1}\leq C(1+t_n)\epsilon^{N+1} $ for
$t_n\leq\epsilon^{-1}.$ This proves \eqref{error MFE}.
\hfill\end{proof}

\subsection{Almost invariants}\label{subsec:alm inv}
\label{subsec:Conservation} According to the above analysis, we can
rewrite the defect formula \eqref{djk}  as
\begin{equation}\label{djk pote}
\begin{aligned}
&\frac{1}{h^2\phi_2(h^2\omega^2_j)}\tilde{L}_j^{\mathbf{k}}
\zeta_j^{\mathbf{k}}+\nabla_{-j}^{-\mathbf{k}}\mathcal{U}(\xi(t))=d_j^{\mathbf{k}},
\end{aligned}
\end{equation}
where $\nabla_{-j}^{-\mathbf{k}}\mathcal{U}(y)$ is the partial
derivative with respect to $y_{-j}^{-\mathbf{k}}$  of the extended
potential (see, e.g. \cite{Cohen08-1,Hairer08})
\begin{equation*}
\begin{aligned}
&\mathcal{U}(\xi(t,\sigma))=\sum\limits_{l=-N}^N\mathcal{U}_l(\xi(t,\sigma)),\\
&\mathcal{U}_l(\xi(t,\sigma))=\sum\limits_{m= 2}^N
\frac{U^{(m+1)}(0)}{(m+1)!}
\sum\limits_{\mathbf{k}^1+\cdots+\mathbf{k}^{m+1}=0}\sum\limits_{j_1+\cdots+j_{m+1}=2Ml}'
\int_{0}^{1}\big(\xi_{j_1}^{\mathbf{k}^1}\cdot\ldots\cdot
\xi_{j_{m+1}}^{\mathbf{k}^{m+1}}\big)(t,\sigma)d\sigma.
\end{aligned}
\end{equation*}

Following \cite{Cohen08-1},   define
$S_{\boldsymbol{\mu}}(\theta)y=\big(\mathrm{e}^{\mathrm{i}(\mathbf{k}
\cdot \boldsymbol{\mu}) \theta}y_j^{\mathbf{k}}\big)_{| j| \leq M,
\norm{\mathbf{k}}\leq K}$ and
$T(\theta)y=\big(\mathrm{e}^{\mathrm{i}j
\theta}y_j^{\mathbf{k}}\big)_{| j| \leq M, \norm{\mathbf{k}}\leq
K},$ where $\boldsymbol{\mu} = (\mu_l)_{l\geq0}$ is an arbitrary
real sequence for $\theta\in R.$ According to the results given in
\cite{Cohen08-1}, we obtain $
\mathcal{U}(S_{\boldsymbol{\mu}}(\theta)y)=\mathcal{U}(y)$ and
$\mathcal{U}_0(T(\theta)y)=\mathcal{U}_0(y)$ for $\theta\in
\mathbb{R}$. Therefore,
\begin{equation}\label{du}
\begin{aligned}
&0=\frac{d}{d\theta}\mid_{\theta=0}
\mathcal{U}(S_{\boldsymbol{\mu}}(\theta)\xi(t,\sigma)),\ \
0=\frac{d}{d\theta}\mid_{\theta=0}
\mathcal{U}_0(T(\theta)\xi(t,\sigma)).
\end{aligned}
\end{equation}

\begin{prop}
There exit two functions
$\mathcal{J}_{l}[\boldsymbol{\zeta},\boldsymbol{\eta}](\tau)$ and
$\mathcal{K}[\boldsymbol{\zeta},\boldsymbol{\eta}](\tau)$ such that
\begin{equation}\label{invariant1}
\begin{aligned}
&\sum\limits_{l=1}^M \omega_l^{2s+1} \left|\frac{d}{d \tau}
\mathcal{J}_{l}[\boldsymbol{\zeta},\boldsymbol{\eta}](\tau)\right|\leq C\epsilon^{N+1},\\
&\left|\frac{d}{d \tau}
\mathcal{K}[\boldsymbol{\zeta},\boldsymbol{\eta}](\tau)\right|\leq
C(\epsilon^{N+1}+\epsilon^{2}M^{-s+1}),
\end{aligned}
\end{equation}
where  $\tau\leq 1.$   Moreover, it is true that
\begin{equation}\label{invariant2}
\begin{aligned}
&  \mathcal{J}_{l}[\boldsymbol{\zeta},\boldsymbol{\eta}](\epsilon
t_n)=\hat{J}_{l}(\mathbf{q}_{n} ,
\mathbf{p}_{n})+\gamma_l(t_n)\epsilon^3,\\
&\mathcal{K}[\boldsymbol{\zeta},\boldsymbol{\eta}](\epsilon
t_n)=\hat{K}(\mathbf{q}_{n} ,
\mathbf{p}_{n})+\mathcal{O}(\epsilon^3)+\mathcal{O}(\epsilon^2
M^{-s}),
\end{aligned}
\end{equation}
where $$\hat{J}_{l} = \hat{I}_l + \hat{I}_{-l} = 2\hat{I}_l\quad
\textmd{for} \quad 0 < l < M,\quad \hat{J}_0 = \hat{I}_0,\quad
\hat{J}_M = \hat{I}_M.$$ Here all the constants are independent of
$\epsilon, M, h$, and $n$, and $\sum\limits_{l=0}^M \omega_l^{2s+1}
\gamma_l(t_n)\leq C$ for $t_n\leq \epsilon^{-1}.$
\end{prop}
\begin{proof}
 $\bullet$ \textbf{Proof of \eqref{invariant1}.}
From the first formula of \eqref{du}, it follows that
\begin{equation*}
\begin{aligned}
&0=\frac{d}{d\theta}\mid_{\theta=0}
\mathcal{U}(S_{\boldsymbol{\mu}}(\theta)\xi(t,\sigma))
=\sum\limits_{\norm{\mathbf{k}}\leq K}\sum\limits_{|j|\leq M}'
\mathrm{i}(\mathbf{k} \cdot \boldsymbol{\mu})
\xi_{-j}^{-\mathbf{k}}(t,\sigma)
\nabla_{-j}^{-\mathbf{k}}\mathcal{U}(\xi(t,\sigma))\\
=&\sum\limits_{\norm{\mathbf{k}}\leq K}\sum\limits_{|j|\leq M}'
\mathrm{i}(\mathbf{k} \cdot \boldsymbol{\mu})
L^{-\mathbf{k}}_4(\sigma)\zeta^{-\mathbf{k}}_{-j}
\Big(\frac{1}{h^2\phi_2(h^2\omega^2_j)}\tilde{L}_j^{\mathbf{k}}
\zeta_j^{\mathbf{k}}-d_j^{\mathbf{k}}\Big).
\end{aligned}
\end{equation*}
Since the right-hand side is independent of $\sigma$, we choose
$\sigma=1/2$ in the following analysis. With  the above formula, we
have
\begin{equation}
\begin{aligned}
&  \sum\limits_{\norm{\mathbf{k}}\leq K}\sum\limits_{|j|\leq M}'
\mathrm{i}(\mathbf{k} \cdot \boldsymbol{\mu}) L^{-\mathbf{k}}_4\big(
\frac{1}{2}\big)\zeta^{-\mathbf{k}}_{-j}
 \frac{1}{h^2\phi_2(h^2\omega^2_j)}\tilde{L}_j^{\mathbf{k}}
\zeta_j^{\mathbf{k}} =  \sum\limits_{\norm{\mathbf{k}}\leq
K}\sum\limits_{|j|\leq M}' \mathrm{i}(\mathbf{k} \cdot
\boldsymbol{\mu})L^{-\mathbf{k}}_4\big(
\frac{1}{2}\big)\zeta^{-\mathbf{k}}_{-j} d_j^{\mathbf{k}} .
\end{aligned}
\label{almost 1-2}%
\end{equation}
By the expansions   of $L^{-\mathbf{k}}_4( \frac{1}{2})$ and
$\tilde{L}_j^{\mathbf{k}}$ and the ``magic formulas" on p. 508 of
\cite{hairer2006}, it is known that the left-hand side of
\eqref{almost 1-2} is a total derivative of function $\epsilon
\mathcal{J}_{\boldsymbol{\mu}}[\boldsymbol{\zeta},\boldsymbol{\eta}](\tau)$
which depends on $\boldsymbol{\zeta}(\tau), \boldsymbol{\eta}(\tau)$
and their up to $(N-1)$th order derivatives. This means that
\eqref{almost 1-2} is identical to
\begin{equation*}
\begin{aligned}
& -\epsilon\frac{d}{d \tau}
\mathcal{J}_{\boldsymbol{\mu}}[\boldsymbol{\zeta},\boldsymbol{\eta}](\tau)
=
 \sum\limits_{\norm{\mathbf{k}}\leq K}\sum\limits_{|j|\leq M}' \mathrm{i}(\mathbf{k}
\cdot \boldsymbol{\mu}) L^{-\mathbf{k}}_4\big(
\frac{1}{2}\big)\zeta^{-\mathbf{k}}_{-j} d_j^{\mathbf{k}} .
\end{aligned}
\end{equation*}
Consider the special case of $\boldsymbol{\mu}=\langle l\rangle$ in
what follows. By letting
$z_j^{\mathbf{k}}=L_4^{\mathbf{k}}(1/2)\zeta^{\mathbf{k}}_{j}$ and
the property of $L_4^{\mathbf{k}}(1/2)$, it is easy to know that the
bounds of $z_j^{\mathbf{k}}$ and $\zeta_j^{\mathbf{k}}$ are of the
same magnitude. Splitting $\mathbf{d} = a\mathbf{d}+ b\mathbf{d}$
into two parts: the diagonal ($\mathbf{k}=\pm\langle j\rangle$) and
nondiagonal ($\mathbf{k}\neq\pm\langle j\rangle$), it is clear that
\begin{equation*}
\begin{aligned}
|||a\mathbf{d} |||_s^2+\sum\limits_{\norm{\mathbf{k}}\leq
K}|||\boldsymbol{\omega}^{s\abs{\mathbf{k}}} b\mathbf{d}
|||_0^2=\sum\limits_{\norm{\mathbf{k}}\leq
K}\norm{\boldsymbol{\omega}^{s\abs{\mathbf{k}}}\mathbf{d}^{\mathbf{k}}}_0^2\leq
C \epsilon^{2N+2},
\end{aligned}
\end{equation*}
where \eqref{alter esti non} is used. According to Lemma 3 of
\cite{Cohen08} and the facts that
$z_j^{\mathbf{k}}=\frac{\epsilon}{\omega_j^s}\hat{a}z_j^{\mathbf{k}}+\frac{\epsilon^{[[\mathbf{k}]]}}{\boldsymbol{\omega}
^{s\abs{\mathbf{k}}}}\hat{a}z_j^{\mathbf{k}}$ and
$|||\hat{a}z^{\mathbf{k}} |||_1\leq C,\
|||\Omega\hat{b}z^{\mathbf{k}} |||_1\leq C$ from \eqref{Bounds zeta
sca}, one arrives at
\begin{equation*}
\begin{aligned}
&\sum\limits_{l=1}^M \omega_l^{2s+1} \left|\frac{d}{d \tau}
\mathcal{J}_{l}[\boldsymbol{\zeta},\boldsymbol{\eta}](\tau)\right|=\frac{1}{\epsilon}\sum\limits_{l=1}^M
\omega_l^{2s+1}   \abs{\sum\limits_{\norm{\mathbf{k}}\leq
K}k_l\sum\limits_{j=-\infty}^{\infty}\zeta_j^{\mathbf{k}}d_j^{\mathbf{k}}}\\
\leq&
\frac{1}{\epsilon}\Big[|||\frac{\epsilon}{\omega_j^s}\hat{a}\zeta_j^{\mathbf{k}}|||_{s+1}|||a\mathbf{d}|||_{s}
+\Big(\sum\limits_{\norm{\mathbf{k}}\leq K}\norm{\boldsymbol{\omega}
^{s\abs{\mathbf{k}}}(1+\abs{k
\cdot\boldsymbol{\omega}})\frac{\epsilon^{[[\mathbf{k}]]}}{\boldsymbol{\omega}
^{s\abs{\mathbf{k}}}}\hat{a}\zeta_j^{\mathbf{k}}}_0^2\Big)^{1/2}\\
&\big(\sum\limits_{\norm{\mathbf{k}}\leq K}\norm{\boldsymbol{\omega}
^{s\abs{\mathbf{k}}}b\mathbf{d}^{\mathbf{k}}}_0^2\big)^{1/2}\Big]\\
\leq& C\epsilon ^{N+1}.
\end{aligned}
\end{equation*}
The first statement of  \eqref{invariant1} is proved.

With the second formula of \eqref{du} and in a similar way, one gets
\begin{equation}
\begin{aligned}
&  \sum\limits_{\norm{\mathbf{k}}\leq K}\sum\limits_{|j|\leq M}'
\mathrm{i}j L^{-\mathbf{k}}_4\big(
\frac{1}{2}\big)\zeta_{-j}^{-\mathbf{k}}
 \frac{1}{h^2\phi_2(h^2\omega^2_j)}\tilde{L}_j^{\mathbf{k}}
\zeta_j^{\mathbf{k}}\\
=& \sum\limits_{\norm{\mathbf{k}}\leq K}\sum\limits_{|j|\leq M}'
\mathrm{i}jL^{-\mathbf{k}}_4\big(
\frac{1}{2}\big)\zeta_{-j}^{-\mathbf{k}} \Big(d_j^{\mathbf{k}}-
\sum\limits_{l\neq0}\nabla_{-j}^{-\mathbf{k}}\big(\mathcal{U}_l(\xi(t,\sigma))\big)\Big).
\end{aligned}
\label{almost 2-2}%
\end{equation}
A careful observation shows that the left-hand side of \eqref{almost
2-2} can be written as a total derivative of function $\epsilon
\mathcal{K}[\boldsymbol{\zeta},\boldsymbol{\eta}](\tau)$, which
yields
\begin{equation}
\begin{aligned}
& -\epsilon\frac{d}{d \tau}
\mathcal{K}[\boldsymbol{\zeta},\boldsymbol{\eta}](\tau)  =
\sum\limits_{\norm{\mathbf{k}}\leq K}\sum\limits_{|j|\leq M}'
\mathrm{i}jL^{-\mathbf{k}}_4\big(
\frac{1}{2}\big)\zeta_{-j}^{-\mathbf{k}} \Big(d_j^{\mathbf{k}}-
\sum\limits_{l\neq0}\nabla_{-j}^{-\mathbf{k}}\big(\mathcal{U}_l(\xi(t,\sigma))\big)\Big).
\end{aligned}
\label{almost 2-3}%
\end{equation}
For the first expression in the right-hand side of this formula, it
follows from the Cauchy--Schwarz inequality and the bound
$\abs{j}\leq\omega_j$ that
\begin{equation*}
\begin{aligned}
&\abs{\sum\limits_{\norm{\mathbf{k}}\leq K} \sum\limits_{\abs{j}
\leq K}'\mathrm{i}jz_{-j}^{-\mathbf{k}}  d_j^{\mathbf{k}}} \leq
\Big(\sum\limits_{\norm{\mathbf{k}}\leq K} \sum\limits_{\abs{j} \leq
K}'\omega_j^2\abs{ z_{j}^{\mathbf{k}}}^2\Big)^{1/2}
\Big(\sum\limits_{\norm{\mathbf{k}}\leq K} \sum\limits_{\abs{j} \leq
K}'
\abs{ d_{j}^{\mathbf{k}}}^2\Big)^{1/2}\\
&\leq C \epsilon \Big(\sum\limits_{\norm{\mathbf{k}}\leq K}
\sum\limits_{\abs{j} \leq
K}'\frac{\omega_j^2}{\boldsymbol{\omega}^{\abs{\mathbf{k}}}}
 \frac{\epsilon^{[[\mathbf{k}]]}}{\epsilon^2} \frac{\boldsymbol{\omega}^{\abs{\mathbf{k}}}}{\epsilon^{[[\mathbf{k}]]}}   \abs{ z_{j}^{\mathbf{k}}}^2\Big)^{1/2}
 \Big(\sum\limits_{\norm{\mathbf{k}}\leq K}
\sum\limits_{\abs{j} \leq K}' \abs{
d_{j}^{\mathbf{k}}}^2\Big)^{1/2}\leq C \epsilon^{N+2}.
\end{aligned}
\end{equation*}
The remaining expression of \eqref{almost 2-3} contains terms of the
form \begin{equation*}
\begin{aligned}
&\sum\limits_{\norm{\mathbf{k}}\leq K}\sum\limits_{|j|\leq M}'
\mathrm{i}jz^{-\mathbf{k}}_{-j}
 \nabla_{-j}^{-\mathbf{k}} \mathcal{U}_l(\xi(t,\sigma))\\
 =&\sum\limits_{m=2}^N\frac{U^{(m+1)}(0)}{m!}
\sum\limits_{\mathbf{k}^1+\cdots+\mathbf{k}^{m+1}=\mathbf{k}}\sum\limits_{j_1+\cdots+j_{m+1}=
  2Ml}'  z_{j_1}^{\mathbf{k}^1} \cdots
z_{j_m}^{\mathbf{k}^m}\cdot
\mathrm{i}j_{m+1}z_{j_{m+1}}^{\mathbf{k}^{m+1}},
\end{aligned}
\end{equation*}
which is the $2Ml$th Fourier coefficient of the function (see
\cite{Hairer08})
\begin{equation*}
\begin{aligned}
&w(x):=\sum\limits_{m=2}^N\frac{U^{(m+1)}(0)}{m!}
\sum\limits_{\mathbf{k}^1+\cdots+\mathbf{k}^{m+1}=\mathbf{k}}\mathcal{P}
z^{\mathbf{k}^1}(x) \cdots \mathcal{P}z^{\mathbf{k}^m}(x)\cdot
\frac{d}{dx} \mathcal{P}z^{\mathbf{k}^{m+1}}(x).
\end{aligned}
\end{equation*}
As shown in the proof of Theorem 5.2 of  \cite{Hairer08}, it can be
confirmed that
   $\norm{w}_{s-1}\leq C \epsilon^3$,
and  the $2Ml$th Fourier coefficient of $w$ is bounded by $C
\epsilon^3\omega^{-s+1}_{2Ml}\leq  C \epsilon^3(2Ml)^{-s+1}.$ In
this way, the second statement of  \eqref{invariant1}  is obtained
by   \eqref{almost 2-3}.

 $\bullet$ \textbf{Proof of \eqref{invariant2}.}
In what follows, only  the second statement of \eqref{invariant2} is
proved since the first one  can be obtained in a similar way.

By the scheme of AAVF method, it is obtained that
$$2h \textmd{sinc} (h\Omega)\tilde{\mathbf{p}}(t) =\tilde{\mathbf{q}}(t+h)-\tilde{\mathbf{q}}(t-h)+\mathcal{O}(h^2),$$
which gives $ \tilde{p}_j(t)=\mathrm{i}\omega_j\big(\eta_j^{\langle
j\rangle}(\epsilon t) \mathrm{e}^{\mathrm{i} \omega_j
t}-\eta_j^{-\langle j\rangle}(\epsilon t) \mathrm{e}^{-\mathrm{i}
\omega_j
t}\big)+\mathcal{O}(h\epsilon^2)+\mathcal{O}(h^3\epsilon^2)$.
Therefore, it is true that $\zeta_j^{\langle
j\rangle}=\frac{1}{2}\big(\tilde{q}_j+\frac{1}{\mathrm{i}\omega_j}\tilde{p}_j\big)+\mathcal{O}(\epsilon^2)$
and $\zeta_j^{-\langle
j\rangle}=\frac{1}{2}\big(\tilde{q}_j-\frac{1}{\mathrm{i}\omega_j}\tilde{p}_j\big)+\mathcal{O}(\epsilon^2).
$  With these results, the construction of  $\mathcal{K}$ is given
below:
\begin{equation*}
\begin{aligned}
 &\mathcal{K}[\boldsymbol{\zeta},\boldsymbol{\eta}](\tau)= \sum\limits_{|j|\leq M}' j\frac{1}{2}\frac{4\epsilon h \sin(\frac{1}{2}h \omega_j)
 \cos(\frac{1}{2}h \omega_j)}{2h^2\phi_2(h^2  \omega_j^2)} \Big(
 |\zeta_{j}^{\langle j\rangle}|^2- |\zeta_{j}^{-\langle j\rangle}|^2\Big)+\mathcal{O}(\epsilon^3)\\
 =& \sum\limits_{|j|\leq M}' j\omega_j \frac{ \cos(\frac{1}{2}h \omega_j)}{ \textmd{sinc}(\frac{1}{2}h \omega_j)}\Big(
 |\zeta_{j}^{\langle j\rangle}|^2- |\zeta_{j}^{-\langle
 j\rangle}|^2\Big)+\mathcal{O}(\epsilon^3)\\
 =& \sum\limits_{|j|\leq M}'\frac{ j\omega_j }{4}\frac{ \cos(\frac{1}{2}h \omega_j)}{ \textmd{sinc}(\frac{1}{2}h \omega_j)}\Big(
 |\tilde{q}_j+\frac{1}{\mathrm{i}\omega_j}\tilde{p}_j|^2-
 |\tilde{q}_j-\frac{1}{\mathrm{i}\omega_j}\tilde{p}_j|^2\Big)+\mathcal{O}(\epsilon^3)\\
= &  \sum\limits_{|j|\leq M}'\frac{ \cos(\frac{1}{2}h \omega_j)}{
\textmd{sinc}(\frac{1}{2}h \omega_j)}\frac{ j\omega_j }{4}4
\frac{1}{\mathrm{i}\omega_j} \tilde{q}_{-j}\tilde{p}_j
+\mathcal{O}(\epsilon^3) \\
=&
\hat{K}(\tilde{\mathbf{q}},\tilde{\mathbf{p}})+\mathcal{O}(\epsilon^3)+\mathcal{O}(\epsilon^2
M^{-s})   =
\hat{K}(\mathbf{q}_{n},\mathbf{p}_{n})+\mathcal{O}(\epsilon^3)+\mathcal{O}(\epsilon^2
M^{-s}),
\end{aligned}
\end{equation*}
where the results \eqref{error MFE} and \eqref{jth TMFE}  are used.
\hfill\end{proof}

\subsection{From short to long time intervals} \label{subsec:fro to}
It is noted that with the analysis presented in this paper, the
statement of Theorem \ref{main theo} can be proved by patching
together many intervals of length $\epsilon^{-1}$ in the same way as
that used in \cite{Cohen08-1,Cohen08}.

\section{Conclusions} \label{sec:conclusions}
In this paper,  we   have shown the long time behaviour of the AAVF
method when applied to semilinear wave equations via spatial
spectral semi-discretizations. This method can exactly preserve the
energy in the semi-discretisation and has a near-conservation of
modified actions and modified momentum over long times. The main
result is proved by developing modulated Fourier expansion of the
AAVF method and showing two almost-invariants  of the modulation
system.

The main result of this paper explains rigorously the good long-time
behaviour of EP methods in the numerical treatment of semilinear
wave equations. The  analysis for multi-dimensional wave equations
will be  considered. It is also noted that long-term analysis of
many different methods except EP methods has been given recently for
Schr\"{o}dinger equations and the reader is referred to
\cite{Cohen12,Gauckler17,Gauckler10,Gauckler10-1}. Our another work
will be devoted to the long-term analysis of energy-preserving
methods for solving  Schr\"{o}dinger equations. We are hopeful of
obtaining near-conservation of actions, momentum and density as well
as exact-conservation of energy for  some EP methods applied to
Schr\"{o}dinger equations.
\section*{Acknowledgements}
The authors are grateful to Professor Christian Lubich for his
helpful comments and discussions on the topic of modulated Fourier
expansions. We also thank him for drawing our attention to the
long-term  analysis of energy-preserving methods, which motives this
paper.

\end{document}